\documentclass[a4paper,reqno]{amsart}

\usepackage[utf8]{inputenc}
\usepackage{amssymb}
\usepackage{enumitem}
\usepackage{mathrsfs}
\usepackage{mathtools}
\usepackage{amsmath}
\usepackage{xcolor}
\usepackage[colorlinks=true]{hyperref}
\hypersetup{linkcolor=violet,urlcolor=cyan,citecolor=cyan}

\setlist[enumerate]{label=\emph{(\roman*)}}

\usepackage{environ}

\newtheorem{theorem}{Theorem}[section]

\newtheorem{lemma}[theorem]{Lemma}
\newtheorem{proposition}[theorem]{Proposition}

\theoremstyle{definition}

\newtheorem{remark}[theorem]{Remark}
\numberwithin{equation}{section}
\newcommand{\R}{\mathbb{R}}

\def \C {{\mathbb{C}}}

\def \del {\partial}

\begin{document}

\parindent=0pt

	\title[Cubic Dirac with large data]
	{Cubic Dirac equations with a class of large data}
	
		\author[S.~Dong]{Shijie Dong}
	\address{Southern University of Science and Technology, SUSTech International Center for Mathematics, and Department of Mathematics, 518055 Shenzhen, China.}
\email{dongsj@sustech.edu.cn, shijiedong1991@hotmail.com}
	
		\author[K.~Li]{Kuijie Li}
\address{Nankai University, School of Mathematical Sciences and LPMC, Tianjin, 300071, China.}
\email{kuijiel@nankai.edu.cn}

	\author[J. Zhao]{Jingya Zhao}
	\address{Southern University of Science and Technology, Department of Mathematics, 518055 Shenzhen, China.}
	\email{12231282@mail.sustech.edu.cn}
	\begin{abstract}
		We are interested in massless cubic Dirac equations in two and three space dimensions, known as the Soler model. The solution to this model is known as a wave function, which has the unit $L^2$ norm. We aim to show global existence and asymptotic behavior for the cubic Dirac model with a class of initial data that can be large in $L^2$. 
		
	\end{abstract}
	\maketitle
	\tableofcontents
\section{Introduction}
\subsection{Model problem}

We consider the nonlinear Dirac equation with a cubic nonlinearity in this article, a cornerstone in the realm of relativistic quantum mechanics pivotal for modeling Dirac fermion self-interaction. 
The equations of motion read
\begin{equation}\label{eq:Soler}
-i\gamma^\mu \del_\mu \psi = (\psi^* H \psi) F \psi,
\end{equation}
with 
\begin{align}
 H^* = H, \qquad\qquad \gamma^0 F=F^* \gamma^0. \label{eq:H-F}
 \end{align}
 We use $A^*$ to denote the conjugate transpose of a matrix $A$ and the restrictions on $H$ and $F$ in \eqref{eq:H-F} are to ensure the $L^2$ norm of the solution $\psi$ is conserved.
In the above, the unknown variable $\psi$ takes values in $\mathbb{C}^{2^{d-1}}$ in $\mathbb{R}^{d+1}$ spacetime, where $d=2, 3$. The Dirac matrices $\left\{\gamma^{0}, \cdots,\gamma^{2^{d-1}}\right\}$ are $2^{d-1}\times 2^{d-1}$ matrices in $\R^{d+1}$ satisfying 
\begin{equation}\label{eq:gamma}
\aligned
&\gamma^{\mu}\gamma^{\nu}+\gamma^{\nu}\gamma^{\mu}=-2\eta_{\mu\nu}I,
\\
&(\gamma^{\mu})^{*}=-\eta_{\mu\nu}\gamma^{\nu},
\endaligned
\end{equation}
in which $\eta$ is  the Minkowski metric with signature of $(-, +, \cdots, +)$ and $I$ is the identity matrix. 
The initial data are posed at $t=t_0$
\begin{equation}\label{eq:ID}
\psi(t_0) = \psi_0.
\end{equation}

A quintessential representative within such frameworks is the Soler model ($H=\gamma^0, F=I$),  which encapsulates fermion behavior through self-interaction.
Recall that the solution $\psi$ to a Dirac equation is a wave function of a particle, which maintains a unit $L^2$ norm, i.e., $\|\psi\|_{L^2} = 1$. This motivates us to treat initial data $\psi_0$ in \eqref{eq:ID} with bounded $L^2$ norm, say 2. 
To make our argument work, we need some smallness assumptions when $\psi_0$ is hit with derivatives.
We now discuss the admissible class of large initial data, which are supposed to have bounded $L^2$ norm while other norms might be small.
The ensuing restrictions (below $\epsilon$ is a small parameter) on the initial data, read as
\begin{equation}\label{initialcondi}
\left\{
\begin{aligned}
\|\psi_0 \| & < 2,  \\
\sum_{0\leq |I| \leq N} \| \langle x \rangle^{|I|} \nabla^{I} \psi_0 \| & < 20,\\
\sum_{0\leq |J| \leq N-1} \| \langle x \rangle^{|J|} \nabla \nabla^J \psi_0 \| & < \epsilon, 
\end{aligned}
\right.
\end{equation}
where we use $\|\cdot\|$ to denote the $L^2_{x}$ norm.
A class of non-trivial examples of the initial data are functions of the form
$$
\psi_0(x) =\Big( {\phi_\epsilon(x) + \epsilon \phi( x) \over  \| \phi_\epsilon(x) + \epsilon \phi( x) \|} , 0, \cdots, 0\Big),
$$
in which $\phi$ is any smooth function that decays sufficiently fast at spatial infinity, and $\phi_\epsilon(x) = \epsilon^{d/2} \phi(\epsilon x)$ which preserves its $L^2$ norm. This $\psi_0$ has a unit $L^2$ norm, while its weighted higher order norms are small of order $\epsilon$. It is worth mentioning that with this class of examples as initial data one cannot apply a re-scaling argument to treat the problem.
\\

Within this paper, we elucidate the large initial data set and establish  global existence, time decay, and scattering of solutions to both 2D and 3D nonlinear cubic Dirac equations. The main results are listed as follows.

\begin{theorem}[Global existence in 3D: cubic case]\label{thm:existence}
Consider the cubic Dirac model \eqref{eq:Soler}  in $\R^{3+1}$ spacetime  with $H^* = H, \gamma^0 F=F^* \gamma^0$, and let $N \geq 7$ be an integer. There exists an $\epsilon_0 > 0$, such that for all initial data satisfying the boundedness and smallness conditions
\begin{equation}\label{eq:thm-3Ddata11}
\aligned
\sum_{0\leq |I| \leq N} \| \langle x \rangle^{|I|} \nabla^{I} \psi_0 \| < 20,
\\
\sum_{0\leq |J| \leq N-1} \| \langle x \rangle^{|J|} \nabla \nabla^J \psi_0 \| < \epsilon \leq \epsilon_0,
\endaligned
\end{equation}
the Cauchy problem \eqref{eq:Soler}--\eqref{eq:ID} admits a global solution $\psi$, which decays according to
\begin{equation}\label{eq:thm-3ddecay11}
\aligned
|\psi(t, x)| &\leq C  (1+ t+|x|)^{-1} (1+ |t-|x||)^{-1/2},
\endaligned
\end{equation}
with $C$ a constant.
In addition, the solution $\psi$ scatters linearly.
\end{theorem}

\begin{theorem}[Global existence in 2D: cubic case]\label{thm:existence2d}
Consider the cubic Dirac model \eqref{eq:Soler} in $\R^{2+1}$ spacetime with $H^*=H$  and $F=I$, and let $N \geq 7$ be an integer. There exists an $\epsilon_0 > 0$, such that for all initial data satisfying the boundedness and smallness conditions
\begin{equation}\label{eq:thm-2Ddata}
\aligned
\sum_{0\leq |I| \leq N} \| \langle x \rangle^{|I|} \nabla^{I} \psi_0 \| < 20,
\\
\sum_{0\leq |J| \leq N-1} \| \langle x \rangle^{|J|} \nabla \nabla^J \psi_0 \| < \epsilon \leq \epsilon_0,
\endaligned
\end{equation}
the Cauchy problem \eqref{eq:Soler}--\eqref{eq:ID} admits a global solution $\psi$, which decays as
\begin{equation}\label{eq:thm-2ddecay}
\aligned
|\psi(t, x)| &\leq C (1+ t+|x|)^{-1/2} (1+ |t-|x||)^{-1/2},
\endaligned
\end{equation}
with $C$ a constant.
In addition, the solution $\psi$ scatters linearly in the case $H=\gamma^0$.
\end{theorem}

Furthermore, in the three dimensional space case, we can also treat the nonlinear Dirac equations with quadratic nonlinearity, which is mathematically more challenging. The equations read
\begin{equation}\label{eq:3Dquadra}
\left\{
\begin{aligned}
	-i \gamma^{\mu}\partial_{\mu} \psi & = (\psi^*\gamma^0\psi)e, \ \ \  (x,t)\in \R^{3+1} \\
	\psi|_{t=t_0} & = \psi_0(x),
\end{aligned}
\right.
\end{equation}
where $e$ is any fixed constant vector in $\R^4$; see also \cite{LiZa21b, Zh23}.

\begin{theorem}[Global existence in 3D: quadratic case]\label{thm:3Dquadratic}
Consider the quadratic Dirac model \eqref{eq:3Dquadra}  in $\R^{3+1}$ spacetime, and let $N \geq 7$ be an integer. There exists an $\epsilon_0 > 0$, such that for all initial data satisfying the boundedness and smallness conditions
\begin{equation}\label{eq:thm-3Ddata}
\aligned
\sum_{0\leq |I| \leq N} \| \langle x \rangle^{|I|} \nabla^{I} \psi_0 \| < 20,
\\
\sum_{0\leq |J| \leq N-1} \| \langle x \rangle^{|J|} \nabla \nabla^J \psi_0 \| < \epsilon \leq \epsilon_0,
\endaligned
\end{equation}
the Cauchy problem \eqref{eq:3Dquadra} admits a global solution $\psi$, which decays according to
\begin{equation}\label{eq:thm-3ddecay}
\aligned
|\psi(t, x)| &\leq C  (1+ t+|x|)^{-1} (1+ |t-|x||)^{-1/2},
\endaligned
\end{equation}
with $C$ a constant. In addition, the solution $\psi$ scatters linearly.
\end{theorem}

\begin{remark}
There exist many large data results for wave-type equations, among which we only mention here those that are related to the short pulse data introduced by Christodoulou in \cite{Christodoulou09}. For this class of large data, global existence has been built for several wave-type equations; see for instance \cite{MiPeYu} and the references therein. We want to mention that the short pulse data do not include our large data \eqref{eq:thm-3Ddata11} or \eqref{eq:thm-2Ddata} as a subset; see Section \ref{subsec:dis} for a comparison.
\end{remark}

\begin{remark}
Compared with the  results in \cite{DoLi22,Zh23}, where the initial data is assumed to be compactly supported and small, the initial data in our work can not have a compact support of small size. Otherwise, we have smallness in $L^2$ by Poincar\'e inequality.  In this article, we address both 2D and 3D nonlinear  Dirac equations, and show the global existence, optimal time decay and scattering of solution evolving from a class of large initial data.  
\end{remark}


We have a neat form for Theorems \ref{thm:existence}-\ref{thm:existence2d}, which can be regarded as a consequence of these two theorems.
\begin{theorem}\label{thm:neat}
Consider the cubic Dirac model \eqref{eq:Soler} in $\R^{d+1}$ spacetime with $H^*=H, F= I$, and let  $d=2, 3$, and $N \geq 9$ be an integer. There exists an $\epsilon_0 > 0$, such that for all initial data satisfying the boundedness and smallness conditions
\begin{equation}\label{eq:thm-ID-23}
\aligned
\sum_{0\leq |I| \leq N} \| \langle x \rangle^{|I|} \nabla^{I} \psi_0 \| < 20,
\\
\|  \nabla^N \psi_0 \| < \epsilon \leq \epsilon_0,
\endaligned
\end{equation}
the Cauchy problem \eqref{eq:Soler}--\eqref{eq:ID} admits a global solution $\psi$, which scatters linearly if  $H = \gamma^0$.
\end{theorem}

\begin{remark}
Compared with the  conditions on the initial data in Theorems \ref{thm:existence}-\ref{thm:existence2d}, the restrictions in \eqref{eq:thm-ID-23} seem weaker, but we will see from Proposition \ref{prop:ID-small} that they are almost equivalent. Thus the proof for Theorem \ref{thm:neat} is essentially similar to the proofs for Theorems \ref{thm:existence}-\ref{thm:existence2d}. The reason why we present  Theorems \ref{thm:existence}-\ref{thm:existence2d} first and in details is that these two theorems are easier to read and their proofs are in a neater form.
\end{remark}

\subsection{Brief discussion on related results}
The nonlinear cubic Dirac equations have been extensively studied in the past years. On the one hand, there is a large amount of research regarding the local and global existence of the solution in the low regularity setting. Noting that if $\psi$  solves equation \eqref{eq:Soler}, $\psi_{\lambda}(t,x) = \lambda^{\frac{1}{2}} \psi(\lambda t, \lambda x)$ is also a solution to \eqref{eq:Soler}, thus the cubic Dirac equation \eqref{eq:Soler} has a scale invariance regularity $s_c= \frac{d-1}{2}$, which means $\dot{H}^{s_c}$ is the critical space. The criticality is important in the sense that in general  it is believed that an equation is ill-posed for data in supercritical space $\dot{H}^s$ with $s<s_c$. In terms of the well-posedness result, Bournaveas and Candy \cite{BoCa16} established global existence and scattering for both 2D and 3D cubic Dirac equations \eqref{eq:Soler} with small initial data in critical spaces, and the data there can be large in $L^2$. For the massive cubic Dirac equations (i.e., there is a term $m \psi$ with $m>0$ added in the left hand side of \eqref{eq:Soler}), global well-posedness and scattering are proved for small initial data at the critical level;  one can refer to  the works by Bejenaru-Herr \cite{BeHe15}-\cite{BeHe16} for the 3D and the 2D cases, respectively. \\

Also, Dirac equations with other types of nonlinearity have attracted a lot of attention. For example, Tesfahun in \cite{Te20} considered a 2D massive nonlinear  Dirac equation with cubic Hartree type nonlinearity ($[V*(\bar{\psi}\psi)]\psi $) and proved global existence and scattering for the solution evolving from small data in $H^s(\R^2)$ with $s>0$; see also \cite{LeeK}. Recently, Tesfahun's global existence result was further improved by Georgiev-Shakarov \cite{Geor22} in two aspects: the first is the elimination of the smallness condition on the initial data; the other  is the unified way to treat both  2D massive and massless cases. However, the scattering of the solution is missing in this result.
\\

On the other hand,  there is also some research on the nonlinear Dirac equations in the high regularity setting,  aiming at obtaining  global existence and long time behavior of the solution, including pointwise decay estimate and scattering.   Recall that  Dong and Li investigated the 2D cubic Dirac equation with a mass parameter $m\in [0, 1]$,  and proved  uniform-in-mass global existence and unified pointwise estimates for the equations with small high-regular  initial data that are compactly supported;  see \cite{DoLi22}.   This result was extended to 3D and 2D Dirac equations with quadratic nonlinearity by  Zhang  in \cite{Zh23}; see also \cite{LiZa21, KatayamaKubo}.    We remark that the smallness of the initial data  plays a key role in the proofs of the aforementioned results. Motivated by the result in \cite{Geor22}, we also want to study the cubic Dirac equation with some class of large initial data, at least in the $L^2$ level, which is conserved in the evolution of the equation.    
    
\

Last but not least, we briefly mention some large data global existence results for wave-type equations. 
In \cite{LiDuke}, Li developed an original non-local energy bootstrap strategy and settled large data global well-posedness of hedgehog solutions for the Skyrme model in $\mathbb{R}^{3+1}$ which is an energy-supercritical problem.
In \cite{MiPeYu}, Miao-Pei-Yu treated semilinear wave equations under the null condition, and showed global existence for this system with short pulse data of Christodoulou.
There exist other important  large data results on wave-type equations, but we are not going to be exhaustive, and one refers for instance to the aforementioned literature and the references therein.

\subsection{Difficulties and strategies}

In the present paper, one of the most important part is to identify the class of large initial data that can be large in $L^2$, i.e., \eqref{eq:thm-3Ddata11}, \eqref{eq:thm-2Ddata}, and \eqref{eq:thm-ID-23}. With this settled, we next need to show global existence and asymptotic behavior for the nonlinear Dirac equations. In this part, the difficulties mainly arise from the slow decay rates of the solutions and the presence of the large data. We need to carefully balance the decay rates and the smallness of the solutions so that the estimates can be closed.
\\

Our proof to show global existence  is based on Klainerman's vector field method and a bootstrap argument. The restriction condition in \eqref{eq:thm-3Ddata11} (also \eqref{eq:thm-2Ddata}, \eqref{eq:thm-ID-23}), roughly speaking, indicates that the initial data is bounded in $L^2$ itself or when hit by the rotation vector fields (see definition in Section \ref{sec:pre}), and small if it is hit with extra derivatives. After acting vector fields on the solution, we want to show the boundedness and smallness are preserved by the Dirac equations. 
\\

The first difficulty  is to show the smallness of $\|\langle x\rangle^{|I|} \partial \partial^I \psi (t_0)\|$
    where $\partial=(\partial_t, \nabla)$ and $t_0$ is the initial time. To illustrate the idea, we take $\|\partial_t \psi\|$ as an example. Following the equation \eqref{eq:Soler}, we have 
    \begin{align}
    	\|\partial_t \psi(t_0) \| 
\lesssim
\sum_a \|\partial_a \psi(t_0)\| +  \|\psi(t_0)\|_{L^\infty}^2 \|\psi(t_0)\|.
    \end{align}
This requires us to gain some smallness from the term $\|\psi(t_0)\|_{L^\infty}$, and it turns out that this can be realized via the fundamental theorem of calculus (below $r=|x|$)
\begin{align}
	|\psi(t_0, x)| 
&\leq \int_{r}^{+\infty} |\partial_r \psi(t_0, y)| \, d|y|,
\end{align}
after showing smallness and sufficient decay rate on $|\partial_r \psi(t_0, y)|$.
\\

Another difficulty lies in closing the top order energy estimate. Let us take a look at the 3D cubic case (the 3D quadratic case is much harder, and thus in Section \ref{sec:3d} we only illustrate the proof for the quadratic case in 3D), 
\begin{align*}
	E(\widehat{\Gamma}^I \psi, t)^{\frac 12} & \lesssim 1 + \int_{t_0}^t \|\widehat{\Gamma}^I((\psi^*H\psi )F\psi )\| ds \\
	 & \lesssim 1 + \int_{t_0}^t \|\widehat{\Gamma}^I\psi\| \|\psi\|^2_{L^{\infty}} ds + \cdots
\end{align*}
Hence, we need to exploit sufficient decay with smallness for $\|\psi\|_{L^{\infty}}$. 
By the Klainerman-Sobolev inequality (see Proposition \ref{prop:K-S}), we have 
\begin{align}
	|\partial \psi(t,x)| \lesssim \langle t+r \rangle^{-1} \langle t-r \rangle ^{-\frac{3}{2}}.
\end{align}
On the other hand, as a result of the Klainerman-Sobolev inequality, 
\begin{align}
	|\partial \psi(t,x)| \lesssim \epsilon \langle t+r \rangle^{-1} \langle t-r \rangle ^{-\frac{1}{2}}.
\end{align}
By interpolating the above two bounds, we can get suitable estimates for $\partial \psi$, which enjoy some smallness in $\epsilon$ and have sufficiently fast decay,  and which then yields smallness and sufficient decay for $\psi$ via the fundamental theorem of calculus. \\

The 2D case is more subtle compared with the 3D cubic case, since the decay rate of Dirac solution is slower, and here the structure of the nonlinearity will play an important role. We remark that Bournaveas \cite{Boura00} unveiled the hidden null structure of the nonlinear term $\psi^*\gamma^0 \psi$ by introducing a new unknown function $\Psi$ which is defined as the solution to $i\gamma^{\mu}\partial_{\mu} \Psi = \psi$. In this way, $\psi^*\gamma^0\psi$ can be written as a combination of standard null terms\footnote{We call $\partial_t f \partial_t g - \sum_a \partial_a f \partial_a g$ and $\partial_\alpha f \partial_\beta g - \partial_\beta f \partial_\alpha g$ the standard null forms of two functions $f, g$.} 
$$
\psi^*\gamma^0\psi = \left(  (\del_t \Psi)^* \del_t (\gamma^0 \Psi) - (\del_i \Psi)^*  \del^i(\gamma^0\Psi) \right) + \left( -(\del_i \Psi)^*  \del_t(\gamma^i\Psi) + (\del_t \Psi)^*  \del_i(\gamma^i\Psi)\right).
$$
In the present paper, we shall work directly with the Dirac solution $\psi$ and applying the ghost weight energy method due to Alinhac \cite{Alinhac01b},  as such the nonlinear term $\psi^*\gamma^0\psi$ can roughly be bounded by $|[\psi]_{-}|\, |\psi|$, where $[\psi]_{-}$ (see Section \ref{sec:pre}) behaves like the good unknown in wave equation, thus has better decay and spacetime integrability. This is motivated by earlier works of \cite{DongZoeDKG, DoLi22}.\\

\subsection{Further discussions}\label{subsec:dis}

\paragraph{Brief comparison with ``short pulse data".}

\

The ``short pulse data", introduced by Christodoulou in the seminal work \cite{Christodoulou09}, represent an important class of large data, which are actively studied to build blow-up or global existence results for wave-type equations. Here we make a brief comparison between the short pulse data and our large data.

\begin{table}[ht]
\vskip-.25cm
\hskip-.25cm
\centering
\begin{tabular}{|c ||c| c| c| c |}
\hline\hline {\footnotesize Data type} & {\footnotesize Pointwise} & 
{\footnotesize $\|\cdot\|$} &    {\footnotesize $\|\cdot\|_{\dot{H}^1}$} & {\footnotesize $\|\cdot\|_{\dot{H}^2}$ }    
\\
\hline
Short pulse data & $small$ 
& $small$            &  $small$ & $large$ 
 \\
\hline 
Our large data   & $small$ 
&  $large$  & $small$ & $small$ 
\\ [1ex]
\hline
\end{tabular}
\label{table:nonlin}

\caption{Short pulse data and our large data}
\end{table}

By the comparison in the table, we know our large data set is not included in the short pulse data category. We also note that different difficulties arise in studying these two classes of large data in different context.

\

\paragraph{Possible extensions.}

\

The small data global existence for wave-type equations has been investigated extensively in the past few decades. One possible extension is to study  wave-type equations (including Dirac-Klein-Gordon equations, Maxwell-Dirac equations, etc.) with the class of large data \eqref{eq:thm-ID-23} and to  show global existence and asymptotic behavior of the solutions to these models. 

The second possible extension, inspired by \cite{Christodoulou09}, is to explore whether finite time blow-up occurs for some wave-type equations. For this aspect, we expect blow-up most likely to happen in lower spacial dimensions like 2D.

Last but not least, we will consider massive wave-type equations with large data \eqref{eq:thm-ID-23} in the future. Much analysis in the present paper cannot be directly extended to the massive cases due to the presence of the mass terms (for instance one cannot get any smallness on $\|\partial_t \psi(t_0)\|$ if considering $-i\gamma^\mu \del_\mu \psi + \psi = 0$), which makes them challenging problems.

\subsection{Organization} The article is organized as follows. In Section \ref{sec:pre}, we list pertinent notation and present some preliminary estimates, including energy estimates and decay estimates. Sections \ref{sec:3d} and \ref{sec:2d} are devoted  to treating the 3D case and the 2D case of the nonlinear Dirac equations, respectively. Finally, we briefly discuss the proof for Theorem \ref{thm:neat}  in the Section \ref{sec:neat}.

\section{Preliminaries}\label{sec:pre}
\subsection{Notation and conventions}\label{subsec:Notation}

We work in $\R^{d+1}$ spacetime with Minkowski metric $\eta = \text{diag} (-1, 1, \cdots, 1)$ and $d=2,3$.
A spacetime point in $\R^{d+1}$ is denoted by $(x,t)=(x_1,\cdots,x_d, t)$ and $r=|x|$. For simplicity, $\|\cdot\|:= \|\cdot \|_{L_x^2}$, and higher order spacial norms are denoted by the conventional notation.

We recall the following vector fields, which will be frequently used throughout.
\begin{itemize}
	\smallskip
	\item  Translations: \quad $\partial_{\alpha}=\partial_{x_{\alpha}}$, \quad  for $\alpha=0,1,\cdots,d$.
	
	\smallskip
	\item  Lorentz boosts: \quad $L_{a}=x_{a}\partial_{t}+t\partial_{a}$, \quad  for $a=1,\cdots,d$.
	
	\smallskip
	\item  Rotations: \quad  $\Omega_{ab}=x_{a}\partial_{b}-x_{b}\partial_{a}$, \quad  for $1\le a<b\le d$.
	
		\smallskip
	\item   Scaling : \quad  $S=t\partial_{t}+x^{a}\partial_{a}$.
\end{itemize}
We use $\partial $ to denote a spacetime derivative while $\nabla$ is used to represent a spatial derivative.
These vector fields are compatible with wave equations. To treat Dirac equations, following Bachelot \cite{Bache}, we need to modify some of the vector fields so that they commute with the Dirac operator $-i\gamma^\mu \partial_\mu$, which read as follows
\begin{itemize}
	
	\smallskip
	\item  Modified Lorentz boosts: \quad $\widehat{L}_{a}=L_a - {1\over 2} \gamma^0 \gamma^a$, \quad  for $a=1,\cdots,d$.
	
	\smallskip
	\item  Modified rotations: \quad  $\widehat{\Omega}_{ab}=\Omega_{ab} - {1\over 2} \gamma^a \gamma^b$, \quad  for $1\le a<b\le d$.
\end{itemize}
Next, we define the ordered sets of vector fields in $\R^{3+1}$, which can also be adapted to  $\R^{2+1}$ by a slight modification. Let
\begin{align}
	\{\Gamma_i\}_{i=1}^{11} & = \{S, \partial_t, \partial_{1}, \partial_{2}, \partial_{3}, L_1,L_2,L_3, \Omega_{12},\Omega_{13}, \Omega_{23} \},  \\
	\{\widehat{\Gamma}_i\}_{i=1}^{11} & = \{S, \partial_t, \partial_{1}, \partial_{2}, \partial_{3}, \widehat{L}_1, \widehat{L}_2, \widehat{L}_3, \widehat{\Omega}_{12},\widehat{\Omega}_{13}, \widehat{\Omega}_{23} \}.
\end{align}
For any multi-index $I=(i_1, i_2, \cdots, i_{11}) \in \mathbb{N}^{11}$, $|I|= \sum_{k=1}^{11} i_k$, and for another multi-index $J$,  $|I|\leq |J|$ means each component of $I$ is less or equal to the corresponding component in $J$.  We denote 
\begin{align}
	\Gamma^I = \prod_{l=1}^{11} \Gamma_{l}^{i_{l}}, \ \  \qquad  \widehat{\Gamma}^I = \prod_{l=1}^{11} \widehat{\Gamma}_{l}^{i_{l}}.
\end{align}
Since the difference between $\widehat{\Gamma}$ and $\Gamma$ is a constant matrix, for any $L\in \mathbb{N}$ and smooth scalar or vector valued function $g$, we can see 
\begin{align}
	\sum_{|I|\leq L}|\widehat{\Gamma}^{I} g | \lesssim \sum_{|J|\leq L} |\Gamma^J g| \lesssim \sum_{|I|\leq L} |\widehat{\Gamma}^I g|.
\end{align}
Finally, we use $C$ to denote a constant which may be different from line to line. Also, by $A\lesssim B$, we  mean there is a generic constant $C>0$, such that $A \leq C B$. We apply the Japanese bracket to denote $\langle \cdot \rangle = \sqrt{1+|\cdot|^2}$. Throughout the paper, spacetime indices are represented by Greek letters while Roman letters is used to denote spacial indices. For a vector $\Psi$, we define 
$$[\Psi]_\pm = \Psi \pm {x_a\over r} \gamma^0 \gamma^a \Psi.$$
 The Einstein summation convention over repeated upper and lower indices is adopted unless otherwise specified.

\subsection{Commutator estimates} For two operators $A, B$, let $[A, B]= AB-BA$ denote the commutator of $A, B$. We list some well-known commutator  estimates which shall be frequently used in the sequel; see for instance \cite{Sogge, Bache}.

\begin{proposition} \label{prop:commu1}
	There exist some constants $c^{\alpha}$, such that the following estimates hold.
	\begin{align*}
		[\partial, \Gamma] & = c^{\alpha} \partial_{\alpha}, \ \  \qquad   [S, -i\gamma^{\mu}\partial_{\mu}] = i\gamma^{\mu}\partial_{\mu}, \\
		[\widehat{\Gamma}_{i}, -i\gamma^{\mu} \partial_{\mu}] & = 0, \ \quad \  \textrm{for}\ \  i =2, \cdots, 11.
	\end{align*}
\end{proposition}

\begin{proof}
	The above formulas can be obtained by a straightforward calculation, we just compute   $[S, -i\gamma^{\mu}\partial_{\mu}]$, the others can be similarly handled. 
	\begin{align*}
		[S, -i\gamma^{\mu} \partial_{\mu}] & = [t\partial_t, -i\gamma^{\mu} \partial_{\mu}] + [x^j \partial_j, -i\gamma^{\mu} \partial_{\mu}] \\
		& =i \gamma^0 \partial_t + i\gamma^j\partial_j.
	\end{align*}
	The desired result then follows.    
\end{proof}

As a consequence, suppose $\phi$ solves  the Dirac equation
\begin{align}
	-i \gamma^{\mu} \partial_{\mu} \phi = F.
\end{align}
Then for any multi-index $I$, we have 
\begin{align} \label{eq:vectorfield}
	-i \gamma^{\mu} \partial_{\mu} \widehat{\Gamma}^I\phi =  \sum_{|J| \leq |I|} c_{J,I}\widehat{\Gamma}^J F,
\end{align}
where $c_{J,I} $ are constants. 
Considering the special form of the nonlinearities in \eqref{eq:Soler}, we recall the following structure-preserving results.
\begin{lemma}[\cite{DongZoeDKG}]\label{lem:preserve}
Let $f$ be a smooth scalar function, and $\Psi, \Phi$ be two smooth $\C^4$-valued functions in $\R^{3+1}$.
Then the following results hold:
\begin{align*}
\widehat{\Gamma}^I (f \Psi) &= \sum_{I_1+I_2=I} \Gamma^{I_1} f \widehat{\Gamma}^{I_2} \Psi,
\\
\Gamma^I (\Psi^* \gamma^0 \Phi) &= \sum_{I_1+I_2=I} \big(\widehat{\Gamma}^{I_1}\Psi\big)^*  \gamma^0 \widehat{\Gamma}^{I_2}\Phi.
\end{align*}
\end{lemma}
In Lemma \ref{lem:preserve} we only consider the 3D case, and its 2D analogue can be formulated in the same manner which is omitted here.
\\
The following commutator estimates will also be used, which we formulate as a proposition. 
\begin{proposition} \label{prop:commu2}
	For any smooth scalar or vector valued function $f$ and multi-index $I$,  there exists some generic constant $C>0$, such that 
	\begin{align}
		|[\partial_{\alpha}, \Gamma^I] f| & \leq C \sum _{|J|<|I|}\sum_{0 \leq \beta \leq d } |\partial_{\beta} \Gamma^J f|, \\
		|[\partial_{\alpha}, \widehat{\Gamma}^I] f| & \leq C \sum _{|J|<|I|}\sum_{0 \leq \beta \leq d } |\partial_{\beta} \widehat{\Gamma}^J f|, \\
		|[\Gamma_k, \Gamma^I] f| & \leq C \sum_{|J|\leq |I|} |\Gamma^J f|.
	\end{align}
\end{proposition}
\begin{proof}
	One can prove a simple case $|I| =1$ by elementary computation, then an induction argument can give the final result, we omit the details.
\end{proof}

\subsection{Decay estimates}

First we recall the standard Klainerman-Sobolev inequality which gives pointwise decay via weighted $L^2$-type norms.

\begin{proposition}[\cite{Sogge}]\label{prop:K-S}
Let $f$ be a sufficiently regular function in $\R^{d+1}$ with $d=2, 3$. 
It holds that
	\begin{equation}
	|f(t,x)|\lesssim \langle t+r\rangle^{-\frac{d-1}{2}}\langle t-r\rangle^{-\frac{1}{2}}\sum_{|I|\le 2}\left\|\Gamma^{I}f(t,x)\right\|.
	\end{equation}
Consequently, if $\phi$ is a sufficiently regular $\C^{2^{d-1}}$-valued function in $\R^{d+1}$ with $d=2, 3$.
It holds that
	\begin{equation}
	|\phi(t,x)|\lesssim \langle t+r\rangle^{-\frac{d-1}{2}}\langle t-r\rangle^{-\frac{1}{2}}\sum_{|I|\le 2}\left\|\widehat{\Gamma}^{I}\phi(t,x)\right\|.
	\end{equation}
\end{proposition}
Next, we state a classical result, whose proof can be found in \cite{Sogge}. 
\begin{proposition}\label{prop:extradecay}
	Let $u$ be a sufficiently regular function in $\R^{d+1}$ with $d \geq 2$. Then we have 
	\begin{align}
		\langle t- r \rangle |\partial u| \lesssim  |\Gamma u| \lesssim \sum_{|J|\leq 1}|\widehat{\Gamma}^J u|.
	\end{align}
\end{proposition}

The following classical decay estimate requires weaker assumptions on the function and in turn gives weaker decay rate.

\begin{proposition}[\cite{KlainWave}]\label{prop:K-S-2}
Let $f$ be a sufficiently regular function in $\R^{d+1}$ with $d=2, 3$. 
It holds that
	\begin{equation}
	|f(t,x)|\lesssim \langle r\rangle^{-\frac{d-1}{2}}\sum_{\substack{|I_1| + |J_1|\le 2,\\ |I_2|+|J_2|\leq 2}}\left\|\nabla^{I_1} \Omega^{J_1} f(t,x)\right\|^{1/2} \left\|\nabla \nabla^{I_2}\Omega^{J_2}f(t,x)\right\|^{1/2}.
	\end{equation}
Consequently, if $\phi$ is a sufficiently regular $\C^{2^{d-1}}$-valued function in $\R^{d+1}$ with $d=2, 3$.
It holds that
	\begin{equation}
|\phi(t,x)|\lesssim \langle r\rangle^{-\frac{d-1}{2}}\sum_{\substack{|I_1| + |J_1|\le 2,\\ |I_2|+|J_2|\leq 2}}\left\|\nabla^{I_1} \widehat{\Omega}^{J_1} \phi(t,x)\right\|^{1/2} \left\|\nabla \nabla^{I_2}\widehat{\Omega}^{J_2}\phi(t,x)\right\|^{1/2}.	\end{equation}
\end{proposition}
At the end of this subsection, we recall the well-known H\"older's inequality and Hardy's inequality.
\begin{proposition}\label{eq:holder}
	Let $0 <\alpha < 1,  1 \leq p \leq \infty$,  then we have 
	\begin{align}
		\||f|^{\alpha} |g|^{1-\alpha}\|_{p} \leq \|f\|_{p}^{\alpha}\|g\|_{p}^{1-\alpha}.
	\end{align}
	where $\|\cdot\|_{p}$ denotes the standard $L^p$ norm. 
	\end{proposition}
	
	\begin{proposition} \label{ineq:hardy}
		Suppose that $f$ is a smooth function in $\R^3$ and $f\in H^1(\R^3)$, then there exists a universal constant $C>0$, such that 
		\begin{align}
			\bigg \| \frac{f(x)}{|x|} \bigg\| \leq C \|\partial_r f\|.
		\end{align} 
	\end{proposition}

\subsection{Energy estimates for Dirac equations}

Given a function $\phi$, following Alinhac \cite{Alinhac01b}, we define its ghost weight energy (below $\delta>0$)
\begin{equation}\label{eq:DiracE}
E(\phi, t ) = \int_{\R^d} |\phi(t, x)|^2 \, dx + \int_{t_0}^t \int_{\R^d} {| [\phi(\tau, x)]_-|^2\over \langle \tau-r\rangle^{1+2\delta}} \, dxd\tau.
\end{equation}
The following energy estimates allow one benefit from the $\gamma^0$ structure in the nonlinear terms, which is from earlier observations in \cite{DoLiMaYu, DongZoeDKG}.

\begin{proposition}\label{prop:DiracEE}
Let $\phi$ be the solution to the Dirac equation
$$
\aligned
-i\gamma^\mu \partial_\mu \phi = F,
\\
\phi(t_0) = \phi_0,
\endaligned 
$$
in which $F$ is a sufficiently nice function.
Then, one has
\begin{equation}
E(\phi, t ) \lesssim_\delta E(\phi, t_0 ) + \int_{t_0}^t \int_{\R^d} \big| \phi^* \gamma^0 F \big| \, dxds.
\end{equation}

\end{proposition}

\subsection{Pointwise estimates for Dirac equations}
In \eqref{eq:DiracE} and Proposition \ref{prop:DiracEE}, we see $[\phi]_-$ enjoys a good spacetime integral bound. Now we show that the pointwise bound of $[\psi]_-$ is also better than $\psi$ in the sense that it decays faster in time.

\begin{lemma}[\cite{DLY}] \label{lem:D-Point-improve}
Let $\phi$ solve
\begin{align}
-i\gamma^\mu \partial_\mu \phi = F, \label{eq:Dirac-F}
\end{align}
then one has 
\begin{align*}
|\partial [\phi]_- | \lesssim {1\over \langle t\rangle} \sum_{|I|\leq 1} |\widehat{\Gamma}^I \phi|  + |F|,
\qquad {t/2} \leq |x|\leq 2t.
\end{align*}
\end{lemma}
\begin{proof}
We only consider  $t\geq 1$ as the conclusion is obvious for $0\leq t\leq 1$.
We use the relation $\partial_{a}=t^{-1}L_{a}-(x_{a}/t)\del_t $ to rewrite~\eqref{eq:Dirac-F} as
	\begin{equation*}
	\left(I-\frac{x_{a}}{r}\gamma^0\gamma^{a}\right)\partial_{t}\phi
	=\gamma^0\gamma^{a}\left(\frac{x_{a}}{t}-\frac{x_{a}}{r}\right)\del_t \phi-\frac{\gamma^0\gamma^{a}}{t}L_{a}\phi+i\gamma^0 F,
	\end{equation*}
	which yields
	\begin{equation}\label{eq:ptvarphi}
	\left|\partial_{t}\left( [\phi]_{-}\right)\right|
	\lesssim \frac{1}{\langle t\rangle}\left|\Gamma \phi\right|+|F|.
	\end{equation}
	Thus in the region ${t/2} \leq |x|\leq 2t$, employing ~\eqref{eq:ptvarphi} and the identity $\partial_{a}=t^{-1}L_{a}-(x_{a}/t)\del_t $ gives 
	\begin{equation}\label{eq:pavarphi}
	\begin{aligned}
	\left|\partial_{a}\left( [\phi]_{-}\right)\right|
	&\lesssim \frac{1}{t}|L_{a}\left( [\phi]_{-}\right)|+|\del_t \left( [\phi]_{-}\right)|\\
	&\lesssim   \frac{1}{\langle t\rangle}\big(\left|\Gamma \phi\right| + |\phi| \big)+|F|.
	\end{aligned}
	\end{equation}
	By \eqref{eq:ptvarphi} and \eqref{eq:pavarphi}, the proof is done.
\end{proof}

\subsection{Structure of the nonlinearity}

In 2D,  since the decay rate for Dirac solution is slower,  we need to exploit the structure of the nonlinear term $\psi^*\gamma^0 \psi$ so as to close our argument.   
\begin{lemma}[\cite{DongZoeDKG, DoLi22}]\label{lem:decompose}
Let $\Psi$ and $\Phi$ be two $\C^{2^{d-1}}$-valued functions, and recall that
\begin{align*}
[\Psi]_\pm = \Psi \pm {x_a\over r} \gamma^0 \gamma^a \Psi.
\end{align*}
Then we have
\begin{align*}
\Psi^*\gamma^0\Phi = \frac{1}{4} \Big([\Psi]_{-}^* \gamma^0 [\Phi]_{-} + [\Psi]_{-}^* \gamma^0 [\Phi]_+ + [\Psi]_+^* \gamma^0 [\Phi]_{-}   \Big).
\end{align*}
In addition, if $\Psi$ and $\Phi$ are sufficiently smooth, it holds
\begin{align*}
|{\Gamma}^I (\Psi^*\gamma^0\Phi)|
\lesssim
\sum_{|I_1|+|I_2|\leq |I|} \Big(\big| [\widehat{\Gamma}^{I_1}\Psi]_{-}\big| \,\big| \widehat{\Gamma}^{I_2}\Phi\big| + \big|\widehat{\Gamma}^{I_1}\Psi\big| \,\big| [\widehat{\Gamma}^{I_2}\Phi]_{-}\big|   \Big).
\end{align*}
\end{lemma}

\subsection{Criteria for scattering}
Finally, we present a sufficient result which leads to linearly scattering of solution to Dirac equation, see \cite{DoLi22}.
\begin{lemma}\label{lem:scatter}
Let $\phi$ solve (below $d=2,3$, $s\geq 0$)
\begin{align*}
-i\gamma^\mu \partial_\mu \phi = F,
\qquad
\phi(t_0) = \phi_0 \in H^s(\R^{d}),
\end{align*}
and suppose 
\begin{align*}
\int_{t_0}^{+\infty} \| F(\tau) \|_{H^s} \, d\tau < +\infty.
\end{align*}
Then there exists a function $\phi^+ \in H^s(\R^{d})$ such that
\begin{align*}
\| \phi(t) - S(t-t_0)\phi^+ \|_{H^s}
\lesssim
\int_{t}^{+\infty} \|F(\tau)\|_{H^s} \, d\tau
\to 0,
\qquad
\text{as  } t\to +\infty.
\end{align*}
In the above, $S(t) = e^{it(i\gamma^0 \gamma^a \partial_a)}$ is the matrix group propagator, and $S(t-t_0)\phi^+$ solves
\begin{align*}
-i\gamma^\mu \partial_\mu \phi = 0,
\qquad
\phi(t_0) = \phi^+.
\end{align*}
\end{lemma}

\section{The 3D quadratic Dirac equations}\label{sec:3d}

In this section, we study the cubic and quadratic Dirac equation  and show Theorem \ref{thm:existence} and Theorem \ref{thm:3Dquadratic}. The proof of Theorem \ref{thm:3Dquadratic} is typical and can be adapted to prove the cubic case, so in the paragraph below, we only treat the quadratic nonlinear Dirac equation and then illustrate some necessary changes to show the cubic case.

To set reasonable bootstrap assumptions, we need the following lemma, which guarantees boundedness or smallness of the initial data hit with different weighted derivatives.
\begin{lemma}\label{lem:ID-3d}
Let \eqref{eq:thm-3Ddata} hold, then we have
\begin{align}
E(\widehat{\Gamma}^I \psi, t_0 )^{1/2} &\lesssim 1,
\qquad |I|\leq N,
\\
 E(\widehat{\Gamma}^J \partial \psi, t_0 )^{1/2} &\lesssim \epsilon,
\qquad |J|\leq N-1.
\end{align}
\end{lemma}
\begin{proof}

\textbf{Step 1: pointwise bounds of $|\nabla^K \psi(t_0)|$ with $|K|\leq N-3$.}

We apply the Klainerman-Sobolev inequality in Proposition \ref{prop:K-S} on \eqref{eq:thm-3Ddata} to get (one can also use Proposition \ref{prop:K-S-2} to get a weaker decay rate)
\begin{align}
|\nabla \nabla^{K} \psi(t_0, x)| \lesssim \min \{\epsilon \langle r\rangle^{-3/2-|K|}, \langle r\rangle^{-5/2-|K|}\}, 
\qquad |K|\leq N-3. \label{eq:ID-decay}
\end{align}
By the fact $|\nabla^{K}\psi(t_0, r=+\infty)|=0$, the estimate \eqref{eq:ID-decay}, and the fundamental theorem of calculus, one gets
\begin{align}
|\nabla^{K}\psi(t_0, x)| 
&\leq \int_{r}^{+\infty} |\partial_r \nabla^{K}\psi(t_0, y)| \, d|y|  \notag
\\
&\leq \int_{r}^{+\infty} |\nabla \nabla^{K}\psi(t_0, y)| \, d|y|   
\\
&\lesssim \epsilon \int_{r}^{+\infty} \langle |y|\rangle^{-3/2-|K|} \, d|y|
\lesssim \epsilon \langle r\rangle^{-1/2-|K|}.  \notag
\end{align}
Applying again the Klainerman-Sobolev inequality on \eqref{eq:thm-3Ddata} gives
\begin{align}
|\nabla^{K}\psi(t_0, x)| 
\lesssim
\min\{  \epsilon \langle r\rangle^{-1/2-|K|}, \langle r\rangle^{-3/2-|K|}\}, \qquad \ |K|\leq N-3. \label{eq:ID-decay2}
\end{align}

\textbf{Step 2: smallness of $\|\langle x\rangle^{J}\partial_t \nabla^{J}\psi(t_0)\|$ with $|J|\leq N-1$.}

We rewrite equation \eqref{eq:3Dquadra} as
\begin{align}
\partial_t \psi = -\gamma^0 \gamma^a \partial_a \psi + i(\psi^* \gamma^0 \psi)\gamma^0 e, \label{eq:Soler-re}
\end{align}
acting $\nabla^{J}$ we get
$$
\partial_t \nabla^{J} \psi = -\gamma^0 \gamma^a \partial_a \nabla^{J} \psi + i  \nabla^{J} \big((\psi^* \gamma^0 \psi)\gamma^0 e\big).
$$
Thus, triangular inequality gives us
\begin{align}
\|\langle x\rangle^{|J|}\partial_t \nabla^{J}&  \psi(t_0) \| \nonumber 
\lesssim
\|\langle x\rangle^{|J|}\nabla \nabla^{J}\psi(t_0)\| \\
&  +  \sum_{\substack{|J_1|+|J_2|\leq |J|,\\|J_1|\leq |J|,|J_2|\leq N-3}}\|\langle x\rangle^{|J_1|} \nabla^{J_1} \psi(t_0)\| \|\langle x\rangle^{|J_2|}\nabla^{J_2} \psi(t_0)\|_{L^\infty}.  \label{eq:ID-psi-t}
\end{align}

Inserting \eqref{eq:ID-decay2} and \eqref{eq:thm-3Ddata} into \eqref{eq:ID-psi-t} gives us
\begin{align}
\|\langle x\rangle^{|J|}\partial_t \nabla^J \psi(t_0) \| 
\lesssim \epsilon.  \label{eq:ID-psi-tx}
\end{align}

\textbf{Step 3: boundedness of $\|\langle x \rangle^{|J|+1} \partial_t \nabla^{J} \psi(t_0)\|$ with $|J| \leq N-1.$} \\
We perform $\nabla^J$ to both sides of \eqref{eq:Soler-re}, then apply the triangular inequality to obtain
\begin{align*}
	\|\langle x \rangle^{|J|+1}\partial_t \nabla^{J} \psi(t_0)\|    
	\lesssim \|\langle x \rangle^{|J|+1} \nabla \nabla^{J} \psi(t_0)\| 
	 + \|\langle x \rangle^{|J|+1}\nabla^J\big((\psi^* \gamma^0 \psi)\gamma^0 e\big)(t_0)\|.
\end{align*}
Following from \eqref{eq:thm-3Ddata} and \eqref{eq:ID-decay2}, we can see
\begin{align}
	\|\langle x \rangle^{|J|+1}\partial_t \nabla^{J} \psi(t_0)\|  \lesssim 1. 
\end{align}

\textbf{Step 4: pointwise bounds of $|\partial_t \nabla^K \psi(t_0)|$ with $|K|\leq N-4$.}

We act $\nabla^K$ to \eqref{eq:Soler-re} to get
\begin{align*}
\partial_t \nabla^{K} \psi = -\gamma^0 \gamma^a \partial_a \nabla^{K} \psi + i \nabla^{K} \big((\psi^* \gamma^0 \psi)\gamma^0 e\big).
\end{align*}
Applying the triangular inequality deduces
\begin{align*}
|\partial_t \nabla^{K} \psi|
\lesssim
| \nabla \nabla^{K} \psi| + \big|\nabla^{K} \big((\psi^* \gamma^0 \psi)\gamma^0 e\big)\big|.
\end{align*}
Then the pointwise estimates in \eqref{eq:ID-decay2}  yield
\begin{align}
|\partial_t \nabla^{K} \psi(t_0)|
\lesssim
\min\{  \epsilon \langle r\rangle^{-1/2-|K|-1}, \langle r\rangle^{-3/2-|K|-1}\}. \label{eq:ID-decay4}
\end{align}

\textbf{Step 5: Conclusion.} \\
Note that we have obtained 
\begin{align*}
	\|\langle x \rangle^{|I|+1} \partial \nabla^I \psi(t_0)\|  & \lesssim 1, \ \ \forall \ |I| \leq N-1,  \\
	\|\langle x \rangle^{|J|}\partial \nabla^J \psi(t_0) \|  & \lesssim \epsilon, \ \ \forall \ |J|\leq N-1, \\
	\big|\partial \nabla^K \psi(t_0) \big| & \lesssim \min\{  \epsilon \langle r\rangle^{-1/2-|K|-1}, \langle r\rangle^{-3/2-|K|-1}\}, \ \forall \ |K|\leq N-4.
\end{align*}
Then one can repeat the above procedure and apply  an induction argument to show 
\begin{align}
	\|\langle x \rangle^{|I|} \partial^I \psi(t_0)\| & \lesssim 1,  \ \ \forall \ |I|\leq N, \\
	\|\langle x \rangle^{|J|} \partial \partial^J \psi(t_0)\| & \lesssim \epsilon, \ \ \forall \  |J|\leq N-1.
\end{align}
 
The proof is completed.
\end{proof}

Considering the Cauchy problem of Dirac equations \eqref{eq:3Dquadra},  we make the following bootstrap assumptions for $t \in [t_0, T)$
\begin{equation}\label{eq:BA3d}
\aligned
E(\widehat{\Gamma}^I \psi, t )^{1/2} &\leq C_1,
\qquad |I|\leq N,
\\
 E(\widehat{\Gamma}^J \partial \psi, t )^{1/2} &\leq C_1 \epsilon,
\qquad |J|\leq N-1,
\endaligned
\end{equation}
	in which $C_1 \gg 1$ is a large constant to be determined, and $\epsilon$ is sufficiently small such that $C_1 \epsilon^{1/4} \ll 1$. In the above, we define
	\begin{equation}\label{eq:T-3d}
	T = \sup \{s: \eqref{eq:BA3d} \text{ holds}\},
	\end{equation}
and clearly $T > t_0$. If $T = +\infty$, then the solution $\psi$ exists globally, and we are done. Our strategy is to first assume $T<+\infty$,  and then to deduce a contradiction, and thus $T$ must be $+\infty$.

	Combined with commutator estimates, we have the following bounds.
\begin{proposition}	\label{prop:L2-3d}
	Let $\psi$ satisfy the estimates in \eqref{eq:BA3d}, then for all $t\in [t_0, T)$ one has
	$$
	\aligned
	&\sum_{|I|\leq N}\| \widehat{\Gamma}^I \psi \| \leq C_1,
	\\
	&\sum_{|J|\leq N-1} \| \widehat{\Gamma}^J \partial \psi \| + \sum_{|J|\leq N-1} \| \partial\widehat{\Gamma}^J \psi \| \lesssim C_1 \epsilon.
	\endaligned
	$$
	\end{proposition}
	
	By the Klainerman-Sobolev inequality in Proposition \ref{prop:K-S}, we have the following pointwise estimates on $\psi$.
	\begin{proposition}\label{prop:pointwise-3d}
	Let $\psi$ satisfy the estimates in \eqref{eq:BA3d}, then for all $t\in [t_0, T)$ the following holds
	\begin{align}
	&\sum_{|I|\leq N-2}| \widehat{\Gamma}^I \psi | \lesssim C_1 \langle t+r\rangle^{-1} \langle t-r\rangle^{-1/2},  \label{eq:point-3d-1}
	\\
	&\sum_{|J|\leq N-3} | \widehat{\Gamma}^J \partial \psi | + \sum_{|J|\leq N-3} | \partial \widehat{\Gamma}^J \psi | 
	\lesssim C_1 \epsilon \langle t+r\rangle^{-1} \langle t-r\rangle^{-1/2}. \label{eq:point-3d-2}
	\end{align}
	\end{proposition}
	
	\subsection{Improved estimates}
	
	Our goal in this section is to prove improved energy bounds for the solution $\psi$, which are listed in the following proposition.
\begin{proposition}	\label{prop:improved-3d}
	If the estimates in \eqref{eq:BA3d} hold, for all $t\in [t_0, T)$ we have
	\begin{equation}
	\aligned
E(\widehat{\Gamma}^I \psi, t )^{1/2} &\leq {1\over 2} C_1,
\qquad |I|\leq N,
\\
 E(\widehat{\Gamma}^J \partial \psi, t )^{1/2} &\leq {1\over 2} C_1 \epsilon,
\qquad |J|\leq N-1.
\endaligned
	\end{equation}
\end{proposition}

	To prove Proposition \ref{prop:improved-3d}, we need to prepare some bounds on $\psi$. First, we note the pointwise estimates for $\widehat{\Gamma}^I \psi$  (with $|I|\leq N-2$) in Proposition \ref{prop:pointwise-3d} do not enjoy any smallness in $\epsilon$, which causes difficulty in showing Proposition \ref{prop:improved-3d}. We now show that $ \widehat{\Gamma}^I \psi$ (with $|I|\leq N-3$) are actually small in the pointwise sense although they are large measured in $L^2$ norm.
	\begin{lemma} \label{smalldecayesti3d}
	Let $0< \delta \ll {1/8} $ be a small parameter, then we have
	\begin{align}
		\sum_{|I|\leq N-3}| \widehat{\Gamma}^I \psi | \lesssim \min \{ C_1 \epsilon^{1/2} \langle t+r\rangle^{-1+\delta} ,  C_1 \epsilon^{1/4}\langle t+r \rangle^{-1+\delta/2}\langle t-r \rangle^{-1/4}\}.   \label{eq:point-3d-3}
\end{align}
	\end{lemma}
	\begin{proof}
Let $|I|\leq N-3$. Owing to Proposition \ref{prop:extradecay} and \eqref{eq:point-3d-1}, one can see
\begin{align}
	|\partial \widehat{\Gamma}^I \psi|(t,x) \lesssim C_1 \langle t+r \rangle^{-1} \langle t-r \rangle^{-3/2}.
\end{align}
 By interpolating the above bounds with the pointwise bounds in \eqref{eq:point-3d-2}, we get 
 \begin{align}
 	|\partial \widehat{\Gamma}^I \psi|(t,x)& \lesssim C_1 \epsilon^{1/2}\langle t+r \rangle^{-1 } \langle t-r \rangle^{-1} \\
 	    & \lesssim C_1 \epsilon^{1/2}\langle t+r \rangle^{-1+\delta } \langle t-r \rangle^{-1-\delta} .
 \end{align}
Since $\widehat{\Gamma}^I \psi(t, |x|=\infty) = 0$, we can use the fundamental theorem of calculus to obtain
\begin{align}
	|\widehat{\Gamma}^I\psi|(t,x) \leq \int_{|x|}^{+\infty} |\partial_r \widehat{\Gamma}^I\psi(t,y)|d|y| \lesssim C_1 \epsilon^{1/2}\langle t+r \rangle^{-1+\delta },
\end{align}
which combining with \eqref{eq:point-3d-1} implies 
\begin{align}
	|\widehat{\Gamma}^I\psi|(t,x) \lesssim C_1 \epsilon^{1/4}\langle t+r \rangle^{-1+\delta/2}\langle t-r \rangle^{-1/4}.
\end{align}
The proof is completed.
\end{proof}

\begin{lemma}\label{lem:refineghost}
	Let $0<\delta \ll 1/8$ be a small parameter, then we have 
	\begin{align}
		\sum_{|I| \leq N- 3 }\big|[\widehat{\Gamma}^I\psi]_{-} \big| \lesssim C_1 \epsilon^{1/4} \langle t+ r\rangle^{-5/4+\delta/2},  \label{eq:point-3d-6a}
		\\
		\sum_{|J| \leq N- 4 }\big|[\widehat{\Gamma}^J \partial\psi]_{-} \big| \lesssim C_1 \epsilon \langle t+ r\rangle^{-5/4}. \label{eq:point-3d-6b}   
	\end{align}
\end{lemma}
\begin{proof}
\textbf{Step 1: proof of \eqref{eq:point-3d-6a}.}

	Fix $|I| \leq N- 3$. In the region $\{r\geq 2t\} \cup \{r \leq t/2\}$, we can see from \eqref{eq:point-3d-3} that
	\begin{align*}
		|[\widehat{\Gamma}^I \psi]_{-}|(t,x) \lesssim C_1\epsilon^{1/4} \langle t + r\rangle^{-5/4+\delta/2}.
	\end{align*}
Now it suffices to bound $[\widehat{\Gamma}^I\psi]_{-}$ in the domain $\{t/2 \leq r \leq 2t\}$. For this purpose, we apply Lemma \ref{lem:D-Point-improve} to obtain 
\begin{align}
	\big|\partial_{r}[\widehat{\Gamma}^I \psi]_{-} \big|(t,x)  & \lesssim \frac{1}{\langle t +r \rangle} \sum_{|J|\leq 1} |\widehat{\Gamma}^J \widehat{\Gamma}^I \psi| + |\widehat{\Gamma}^I \big( (\psi^*\gamma^0\psi) e \big)|  \nonumber \\
	& \lesssim C_1 \langle t+r\rangle^{-2+\delta/2} \langle t-r\rangle^{-1/2},  \label{eq:ghostdecay1}
\end{align}
in which we use the estimates in Proposition \ref{prop:pointwise-3d} and \eqref{eq:point-3d-3}, and the fact $C_1 \epsilon^{1/4} \ll 1$.
	Interpolating between \eqref{eq:ghostdecay1} and \eqref{eq:point-3d-2}, one can have 
	\begin{align*}
		|\partial_r [\widehat{\Gamma}^I \psi]_{-}| & \lesssim C_1 \epsilon^{1/4}\langle t+r \rangle^{-7/4+3\delta/8} \langle t-r\rangle^{-1/2}  \nonumber \\
		& \lesssim C_1 \epsilon^{1/4}\langle t+r \rangle^{-5/4 + \delta/2} \langle t-r\rangle^{-1-\delta/8}.
	\end{align*} 
	Noticing that $|[\widehat{\Gamma}^I \psi]_{-}|(t, r=2t) \lesssim C_1 \epsilon^{1/4} \langle t+ r\rangle^{-5/4+\delta/2}$, we have 
	\begin{align*}
		\big|[\widehat{\Gamma}^I \psi]_{-}\big|(t,x)  &\leq |\widehat{\Gamma}^I \psi|(t, r=2t) + \int_{2t}^{|x|}|\partial_r [\widehat{\Gamma}^I \psi]_{-}|(t, y) d|y| \nonumber \\
		& \lesssim C_1 \epsilon^{1/4} \langle t+ r\rangle^{-5/4+\delta/2}.
	\end{align*}
	
\textbf{Step 2: proof of \eqref{eq:point-3d-6b}.}	
	
Fix $|J|\leq N-4$. By \eqref{eq:point-3d-2}, we know \eqref{eq:point-3d-6b} holds in the region $\{r\geq 2t\} \cup \{r \leq t/2\}$. Thus, below we only focus on the region 	$\{t/2 \leq r \leq 2t\}$. 
By Lemma \ref{lem:D-Point-improve}, we have
\begin{align}
	\big|\partial_{r}[\widehat{\Gamma}^J \partial \psi]_{-} \big|(t,x)  & \lesssim \frac{1}{\langle t +r \rangle} \sum_{|K|\leq 1} |\widehat{\Gamma}^K \widehat{\Gamma}^J \partial \psi| + |\widehat{\Gamma}^J \partial \big( (\psi^*\gamma^0\psi) e \big)|  \nonumber \\
	& \lesssim C_1 \epsilon \langle t+r\rangle^{-2+\delta/2} \langle t-r\rangle^{-1/2},  \label{eq:ghostdecay2}
\end{align}	
	in which we use the estimates in Proposition \ref{prop:pointwise-3d} and \eqref{eq:point-3d-3}, and the fact $C_1 \epsilon^{1/4} \ll 1$.
	By \eqref{eq:point-3d-2}, we know $|[\widehat{\Gamma}^J \partial \psi]_{-}|(t, r=2t) \lesssim C_1 \epsilon \langle t+ r\rangle^{-3/2}$, we have 
	\begin{align*}
		\big|[\widehat{\Gamma}^J \partial \psi]_{-}\big|(t,x)  &\leq |\widehat{\Gamma}^J \partial \psi|(t, r=2t) + \int_{2t}^{|x|}|\partial_r [\widehat{\Gamma}^J \partial \psi]_{-}|(t, y) d|y| \nonumber \\
		& \lesssim C_1 \epsilon \langle t+ r\rangle^{-5/4}.
	\end{align*}
	
	The proof is done.
\end{proof}

	Next, we use energy estimates to improve the bounds in \eqref{eq:BA3d}. We act $\widehat{\Gamma}^I$ and $\widehat{\Gamma}^J  \partial$ to the equation \eqref{eq:3Dquadra} and apply \eqref{eq:vectorfield}  to get
	\begin{align}
	-i\gamma^\mu \del_\mu \widehat{\Gamma}^I \psi &= \sum_{|L| \leq |I|} c_{L,I}\widehat{\Gamma}^L \big((\psi^* \gamma^0 \psi) e \big), 
	\quad |I|\leq N, \label{eq:Dirac-3d-1}
\\
-i\gamma^\mu \del_\mu \widehat{\Gamma}^J \partial \psi &= \sum_{|K|\leq |J|} c_{K,J}\widehat{\Gamma}^K \partial \big((\psi^* \gamma^0 \psi) e \big), 
\quad |J|\leq N-1. \label{eq:Dirac-3d-2}
	\end{align}
	
	\begin{proposition}	\label{prop:improved-3d-rough}
	If the estimates in \eqref{eq:BA3d} hold, for all $t\in [t_0, T)$ we have
	\begin{align}
	E(\widehat{\Gamma}^I \psi, t ) &\lesssim  1 + C_1^3 \epsilon^{1/4},
\qquad |I|\leq N,  \label{eq:3d-rough1}
\\
 E(\widehat{\Gamma}^J  \partial \psi, t ) &\lesssim \epsilon^2 + C_1^{3} \epsilon^{17/8},
\qquad |J|\leq N-1.  \label{eq:3d-rough2}
	\end{align}

\end{proposition}	
	\begin{proof}
	\textbf{Step 1: Proof of \eqref{eq:3d-rough1}.} Let $|I| \leq N$,
	we apply the energy estimates  in Proposition \ref{prop:DiracEE} on \eqref{eq:Dirac-3d-1} to get for all $t\in [t_0, T)$ that
	\begin{align*}
	E(\widehat{\Gamma}^I \psi, t )
	&\lesssim E(\widehat{\Gamma}^I \psi, t_0 ) 
	+ \sum_{|L| \leq |I|}\int_{t_0}^t \big\| \widehat{\Gamma}^I \psi^* \gamma^0  \widehat{\Gamma}^L \big((\psi^* \gamma^0 \psi)  e\big)  \big\|_{L^1} \, ds
	\\
	&\lesssim
	1 +\sum_{|L| \leq |I|} \int_{t_0}^t \big\| \widehat{\Gamma}^I \psi\big\| \big\| \widehat{\Gamma}^L\big((\psi^* \gamma^0 \psi)  e\big)  \big\| \, ds.
	\end{align*}
	The Leibniz rule and Lemma \ref{lem:decompose} imply that  
	\begin{align*}
	\sum_{|L|\leq |I|}\big\| \widehat{\Gamma}^L\big((\psi^* \gamma^0 \psi)  e\big)  \big\|
	& \lesssim
	\sum_{|L| \leq |I|} \big\| \Gamma^L\big((\psi^* \gamma^0 \psi)  e\big)  \big\|
	\\
	\lesssim
	&\sum_{\substack{|I_1| \leq |I| \\|I_2| \leq N-3}} \bigg\| \frac{ [\widehat{\Gamma}^{I_1} \psi ]_{-}} {\langle s -r \rangle^{\frac{1}{2}+\delta}} \bigg\| 
	     \big\| \langle s -r \rangle^{\frac{1}{2}+\delta}\widehat{\Gamma}^{I_2} \psi \big\|_{L^\infty}  \\
	     & \qquad + \sum_{\substack{|I_2| \leq |I| \\|I_1| \leq N-3}} \big\|  [\widehat{\Gamma}^{I_1} \psi ]_{-} \big\|_{L^{\infty}} 
	     \big\| \widehat{\Gamma}^{I_2} \psi \big\|,
	\end{align*}
and then by 	the $L^2$-type estimates in Proposition \ref{prop:L2-3d} and pointwise estimates in \eqref{eq:point-3d-3} and \eqref{eq:point-3d-6a}, we deduce that
	\begin{align*}
	& \sum_{|L|\leq |I|}\big\| \widehat{\Gamma}^L\big((\psi^* \gamma^0 \psi) e\big)  \big\| \\
	& \lesssim
	 C_1 \epsilon^{1/4} \langle t\rangle^{-3/4+3\delta/2} \sum_{|I_1| \leq |I| } \bigg\| \frac{ [\widehat{\Gamma}^{I_1} \psi ]_{-}} {\langle s -r \rangle^{\frac{1}{2}+\delta}} \bigg\| + C_1^2 \epsilon^{1/4} \langle t\rangle^{-5/4+\delta/2}.
	\end{align*}
	Thus, we have
	\begin{align*}
		E(\widehat{\Gamma}^I \psi, t )
\lesssim 1 & + C_1^3 \epsilon^{1/4} \int_{t_0}^t \langle s\rangle^{-5/4+\delta/2} \, ds \\
    &  + C_1^2 \epsilon^{1/4} \sum_{|I_1| \leq |I| } \bigg(\int_{t_0}^t \bigg\| \frac{ [\widehat{\Gamma}^{I_1} \psi ]_{-}} {\langle s -r \rangle^{\frac{1}{2}+\delta}} \bigg\|^2 ds\bigg)^{1/2} \bigg(\int_{t_0}^t\langle s \rangle^{-3/2+3\delta}ds\bigg)^{1/2}.
	\end{align*}
	Then, we apply the bootstrap assumption \eqref{eq:BA3d} to get 
	\begin{align*}
		E(\widehat{\Gamma}^I \psi, t ) \lesssim  1 + C_1^3\epsilon^{1/4}.
	\end{align*}
	This completes the proof of \eqref{eq:3d-rough1}.
	
		\textbf{Step 2: Proof of \eqref{eq:3d-rough2}.} 
		
	Fix $|J| \leq N-1$. To show the second estimate in Proposition \ref{prop:improved-3d-rough}, we apply Proposition \ref{prop:DiracEE} on \eqref{eq:Dirac-3d-2} to derive for all $t\in [t_0, T)$ that
	\begin{align}
	E(\widehat{\Gamma}^J \partial \psi, t )
	&\lesssim E(\widehat{\Gamma}^J \partial\psi, t_0 ) 
	+ \sum_{|K| \leq |J|}\int_{t_0}^t \big\| \widehat{\Gamma}^J \partial\psi^* \gamma^0  \widehat{\Gamma}^K \partial \big((\psi^* \gamma^0 \psi) e\big)  \big\|_{L^1} \, ds \nonumber
	\\
	&\lesssim
	\epsilon^2 + \sum_{|K| \leq |J|} \int_{t_0}^t \big\| \widehat{\Gamma}^J \partial \psi\big\| \big\| \widehat{\Gamma}^K \partial \big((\psi^* \gamma^0 \psi) e\big)  \big\| \, ds. \label{eq:energyder}
	\end{align}
We rely on the Leibniz rule and Lemma \ref{lem:decompose} to deduce
\begin{align*}
	& \sum_{|K| \leq |J|}\big\| \widehat{\Gamma}^K\partial \big((\psi^* \gamma^0 \psi) e\big)  \big\|
	\\
		\lesssim
	&\sum_{\substack{ |K_1|+|K_2| \leq |J| }} \big\| |[\widehat{\Gamma}^{K_1} \partial \psi]_-|
	     |\widehat{\Gamma}^{K_2} \psi | \big\|
	     \\
	     +&\sum_{\substack{ |K_1|+|K_2| \leq |J| }} \big\| |\widehat{\Gamma}^{K_1} \partial \psi|
	     |[\widehat{\Gamma}^{K_2} \psi]_- | \big\|
	     \\
	     =& A_1 + A_2.
	    \end{align*}
	    
	    Next, we bound $A_1, A_2$ separately.
	    
	    For $A_1$, we have
	    \begin{align}
	    A_1
	    \lesssim
	    &\sum_{\substack{ |K_1|+|K_2| \leq |J| \\ |K_2|\leq N-3}} \bigg\| {[\widehat{\Gamma}^{K_1} \partial \psi]_-\over \langle s-r\rangle^{1/2+\delta}} \bigg\|
	     \big\|\langle s-r\rangle^{1/2+\delta} \widehat{\Gamma}^{K_2} \psi  \big\|_{L^\infty} \nonumber
	     \\
	     + &\sum_{\substack{|K_1|+|K_2| \leq |J| \\ |K_2|\leq N-4}} \big\|\langle r \rangle^{-1/8} \widehat{\Gamma}^{K_1} \psi \big\|  \big\|\langle r \rangle^{1/8} [\widehat{\Gamma}^{K_2} \partial \psi]_-  \big\|_{L^\infty}. \label{ineq:esti00}
	    \end{align}
	    In view of Proposition \ref{eq:holder} and Hardy's inequality(see Proposition \ref{ineq:hardy}),  we can find 
	    \begin{align}
	    \big\|\langle r \rangle^{-1/8} \widehat{\Gamma}^{K_1} \psi \big\| & \leq \big\|\langle r \rangle^{-1} \widehat{\Gamma}^{K_1} \psi \big\|^{1/8} \big\| \widehat{\Gamma}^{K_1} \psi \big\|^{7/8}  \nonumber \\
	    & \lesssim \big\|\partial_r \widehat{\Gamma}^{K_1} \psi \big\|^{1/8} \big\| \widehat{\Gamma}^{K_1} \psi \big\|^{7/8}.   \label{ineq:weightl2}
	    \end{align}
	    Now applying the $L^2$ estimates in Proposition \ref{prop:L2-3d} and \eqref{ineq:weightl2},  one can see 
	    \begin{align} \label{ineq:l2esti2}
	    	 \big\|\langle r \rangle^{-1/8} \widehat{\Gamma}^{K_1} \psi \big\| \lesssim  C_1 \epsilon^{1/8}.
	    \end{align}
	    By inserting the estimates \eqref{ineq:l2esti2} into \eqref{ineq:esti00} and using Lemmas \ref{smalldecayesti3d} and \ref{lem:refineghost}, we get
	    \begin{align}
	    A_1
	    \lesssim
	    &C_1^2 \epsilon^{9/8} \langle s\rangle^{-9/8}  +  C_1 \epsilon^{1/4} \langle s\rangle^{-3/4+3\delta/2} \sum_{\substack{ |K_1| \leq |J| }} \bigg\| {[\widehat{\Gamma}^{K_1} \partial \psi]_-\over \langle s-r\rangle^{1/2+\delta}} \bigg\|. \label{ineq:esti7}
	    \end{align}
	    
	    For $A_2$, we note
	    \begin{align*}
	    A_2
	 \lesssim
	    &\sum_{\substack{ |K_1|+|K_2| \leq |J| \\ |K_2|\leq N-3}} \bigg\| {[\widehat{\Gamma}^{K_1} \psi]_-\over \langle r \rangle^{1/4} \langle s-r\rangle^{1/2}} \bigg\|
	     \big\|\langle r \rangle^{1/4} \langle s-r\rangle^{1/2} \widehat{\Gamma}^{K_2} \partial \psi  \big\|_{L^\infty}
	     \\
	     + &\sum_{\substack{|K_1|+|K_2| \leq |J| \\ |K_2|\leq N-3}} \big\| \widehat{\Gamma}^{K_1} \partial \psi \big\|  \big\|[\widehat{\Gamma}^{K_2} \psi]_-  \big\|_{L^\infty}.
	     \end{align*}
	     Thanks to Proposition \ref{eq:holder} and Hardy's inequality, one can get
	     \begin{align}
	     	\bigg\| {[\widehat{\Gamma}^{K_1} \psi]_-\over \langle r \rangle^{1/4} \langle s-r\rangle^{1/2}}  \bigg\| & \leq \bigg \|{[\widehat{\Gamma}^{K_1} \psi]_{-} \over \langle r \rangle } \bigg\|^{1/4} \bigg \|{[\widehat{\Gamma}^{K_1} \psi]_{-} \over \langle s-r \rangle^{2/3} }  \bigg\|^{3/4}  \nonumber \\
	     	& \lesssim \|\partial_r [\widehat{\Gamma}^{K_1} \psi]_{-} \|^{1/4}\bigg \|{[\widehat{\Gamma}^{K_1} \psi]_{-} \over \langle s-r \rangle^{1/2+\delta} }  \bigg\|^{3/4}.  \label{ineq:esti3}
	     \end{align}
	     Recall that $[\widehat{\Gamma}^{K_1} \psi]_{-} = \widehat{\Gamma}^{K_1} \psi - (x_a/r) \gamma^0\gamma^a \widehat{\Gamma}^{K_1} \psi $, by a direct calculation, one can find
	     \begin{align} \label{ineq:esti4}
	     	\|\partial_r [\widehat{\Gamma}^{K_1} \psi]_{-}\| \lesssim \|\nabla \widehat{\Gamma}^{K_1} \psi\|+ \bigg\|\frac{\widehat{\Gamma}^{K_1} \psi}{r} \bigg\| \lesssim \|\nabla \widehat{\Gamma}^{K_1} \psi\|.
	     \end{align}
	  Inserting \eqref{ineq:esti4} into  \eqref{ineq:esti3}, and using estimates in Proposition \ref{prop:L2-3d}, we get 
	  \begin{align}
	  	\bigg\| {[\widehat{\Gamma}^{K_1} \psi]_-\over \langle r \rangle^{1/4} \langle s-r\rangle^{1/2}}  \bigg\| & \lesssim  C_1^{1/4}\epsilon^{1/4}\bigg \|{[\widehat{\Gamma}^{K_1} \psi]_{-} \over \langle s-r \rangle^{1/2+\delta} }  \bigg\|^{3/4}. \label{ineq:esti5}
	  \end{align}
	   Then \eqref{ineq:esti5}, Proposition \ref{prop:pointwise-3d} and Lemma \ref{lem:refineghost} can imply
   \begin{align}
     A_2 \lesssim C_1^{5/4}\epsilon^{5/4} \langle t \rangle^{-3/4}\sum_{|K_1|\leq |J|}\bigg \|{[\widehat{\Gamma}^{K_1} \psi]_{-} \over \langle s-r \rangle^{1/2+\delta} }  \bigg\|^{3/4} + C_1^2 \epsilon^{5/4}\langle t \rangle^{-5/4+\delta/2}. \label{ineq:esti6}
   \end{align}	    
	    
	    Substituting  the estimates in \eqref{ineq:esti7} and \eqref{ineq:esti6} into \eqref{eq:energyder}, we  get 
	    \begin{align}
	    	E(\widehat{\Gamma}^J \partial \psi, t ) &\lesssim \epsilon^2 + C_1^3 \epsilon^{\frac{17}{8}} +C_1^3\epsilon^{\frac 94} + C_1^2\epsilon^{\frac 54} \sum_{|K_1|\leq |J|}\int_{t_0}^t \langle s \rangle^{-\frac 34+ \frac{3\delta}{2}} \bigg\| {[\widehat{\Gamma}^{K_1} \partial \psi]_-\over \langle s-r\rangle^{\frac 12+\delta}} \bigg\|ds  \nonumber \\
	    	& \qquad \qquad + C_1^{\frac 94}\epsilon^{\frac 94} \sum_{|K_1|\leq |J|} \int_{t_0}^t \langle s \rangle^{-\frac 34} \bigg \|{[\widehat{\Gamma}^{K_1} \psi]_{-} \over \langle s-r \rangle^{1/2+\delta} }  \bigg\|^{\frac 34}ds.
	    \end{align}
	  H\"older's inequality and bootstrap assumption \eqref{eq:BA3d} immediately yield
	  \begin{align}
	  	E(\widehat{\Gamma}^J \partial \psi, t ) &\lesssim \epsilon^2 + C_1^3 \epsilon^{\frac{17}{8}}.
	  \end{align}   		
	
	The proof is completed.
	\end{proof}
	
	\begin{proof}[Proof of Proposition \ref{prop:improved-3d} and global existence] 
	Relying on  Proposition \ref{prop:improved-3d-rough}, we can deduce the estimates in Proposition \ref{prop:improved-3d} by letting $C_1$ sufficiently large and $\epsilon$ very small such that $C_1 \epsilon^{1/8} \ll {1/4}$.
	
	Since the energy is continuous in time, the estimates in Proposition \ref{prop:improved-3d} implies that there exists $\delta_1>0$ such that for all $t\in [t_0, T+\delta_1]$ it holds
	\begin{align*}
E(\widehat{\Gamma}^I \psi, t )^{1/2} &\leq C_1,
\qquad |I|\leq N,
\\
 E(\widehat{\Gamma}^J  \partial\psi, t )^{1/2} &\leq C_1 \epsilon,
\qquad |J|\leq N-1.
	\end{align*}
	We know this contradicts to the definition of $T$ in \eqref{eq:T-3d}. Consequently $T=+\infty$, and thus the solution to the Cauchy problem \eqref{eq:Soler}-\eqref{eq:ID} exists globally as long as $\epsilon$ very small. 
	The pointwise estimates in \eqref{eq:thm-3ddecay} follows from \eqref{eq:point-3d-1}.
	\end{proof}

\subsection{Scattering in 3D}

We next show the solution $\psi$ scatters linearly in $H^N(\R^3)$. By Lemma \ref{lem:scatter}, we only need to bound
\begin{align*}
\int_{t}^{+\infty} \| (\psi^* \gamma^0 \psi) e \|_{H^N(\R^3)} \, d\tau.
\end{align*}
By the estimates in Propositions \ref{prop:L2-3d}, \ref{prop:pointwise-3d} and Lemma \ref{lem:refineghost}, we get
\begin{align*}
& \| (\psi^* \gamma^0 \psi) e \|_{H^N(\R^3)} \\
&\lesssim \sum_{|I_1|\leq N-3} \|[\nabla^{I_1}\psi]_{-}\|_{L^{\infty}} \|\psi\|_{H^N} + \sum_{ \substack{|I_1|+|I_2|\leq N \\ |I_2|\leq N-3}} \bigg\|\frac{[\nabla^{I_1}\psi]_{-}}{\langle \tau -r \rangle^{\frac 12 +\delta}} \bigg\| \|\langle \tau -r \rangle^{\frac 12 +\delta} \nabla^{I_2} \psi\|_{L^{\infty}} \\
& \lesssim C_1^2 \langle \tau \rangle^{-\frac 54+ \frac \delta2} + C_1 \langle \tau \rangle^{-1+\delta} \sum_{|I_1| \leq N} \bigg\|\frac{[\nabla^{I_1}\psi]_{-}}{\langle \tau -r \rangle^{\frac 12 +\delta}} \bigg\|,
\end{align*}
which combining with H\"older inequality further yields
\begin{align*}
\int_{t}^{+\infty} \| (\psi^* \gamma^0 \psi) e \|_{H^N(\R^3)} \, d\tau
\lesssim C_1^2 \langle t\rangle^{-\frac{1}{4}+\delta} \lesssim C_1^2 \langle t\rangle^{-\frac{1}{8}},
\end{align*}
where we recall that $0<\delta \ll 1/8$. Based on this and Lemma \ref{lem:scatter}, we have the following scattering result.

\begin{proposition}
The solution $\psi$ in Theorem \ref{thm:3Dquadratic} scatters linearly in 3D. More precisely, there exist $\psi^+ \in H^N$ and constant $\widehat{C}$ such that 
\begin{align*}
\|\psi(t) - S(t-t_0) \psi^+\|_{H^N}
\leq
\widehat{C}\langle t\rangle^{-\frac 18}.
\end{align*}

	\end{proposition}
	
\subsection{Remarks on 3D cubic Dirac equation}
	In the previous part, we have obtained the global existence and scattering of the Dirac equation with quadratic nonlinearity. The cubic Dirac equation is relatively easy to handle, as the nonlinear term enjoys more decay and smallness(in the $L^{\infty}$ sense). Now we illustrate the main steps in proving Theorem \ref{thm:existence}.
	
	First, based on the assumption of initial data \eqref{eq:thm-3Ddata11},  one can similarly get 
	\begin{align*}
		E(\widehat{\Gamma}^I \psi, t_0 )^{1/2} &\lesssim 1,
\qquad |I|\leq N,
\\
 E(\widehat{\Gamma}^J \partial \psi, t_0 )^{1/2} &\lesssim \epsilon,
\qquad |J|\leq N-1.
	\end{align*} 
Then we can set the bootstrap assumption as follows.  There exists some $T>0$, such that  for all $t_0\leq t <T$, it holds 
\begin{equation} \label{eq:3dcubicboot}
	\aligned
E(\widehat{\Gamma}^I \psi, t )^{1/2} &\leq C_1,
\qquad |I|\leq N,
\\
 E(\widehat{\Gamma}^J \partial \psi, t )^{1/2} &\leq C_1 \epsilon,
\qquad |J|\leq N-1.
\endaligned
\end{equation}
It suffices to show $T=+\infty$. As before, we argue by contradiction and then refine the estimates in \eqref{eq:3dcubicboot}. In this process, $L^{\infty}$ estimates for lower order derivatives of $\psi$ will play a key role. Following from Proposition \ref{prop:pointwise-3d} and  Lemma \ref{smalldecayesti3d}, one can see
\begin{align*}
 \sum_{|I|\leq N-3}  \big|\partial\widehat{\Gamma}^I \psi\big| &\lesssim C_1 \epsilon \langle t+ r \rangle^{-1}\langle t-r\rangle^{-1/2},   \\
	\sum_{|I|\leq N-3} \big|\widehat{\Gamma}^I \psi\big| &\lesssim C_1 \epsilon^{1/2} \langle t+ r \rangle^{-3/4}. 
\end{align*} 
This suffices to refine the estimates in \eqref{eq:3dcubicboot}, we omit the  details.

\section{The 2D cubic Dirac equations}\label{sec:2d}

	In the 2D case, due to the slow decay nature of free Dirac solutions we only show global existence for the model \eqref{eq:Soler} with $H$ being Hermitian and $F = I$; if additionally $H=\gamma^0$, linear scattering for $\psi$ is obtained. We recall that in 2D the solution to \eqref{eq:Soler} might blow up in finite time even for small smooth initial data if no restrictions are posed on $H, F$. First, we derive the energy bounds at the initial time which serves as the basis of our bootstrap setting.

\begin{lemma}\label{lem:ID-2d}
Let \eqref{eq:thm-2Ddata} hold, then we have
\begin{align}
E(\widehat{\Gamma}^I \psi, t_0 )^{1/2} &\lesssim 1,
\qquad |I|\leq N,
\\
 E(\widehat{\Gamma}^J  \partial \psi, t_0 )^{1/2} &\lesssim \epsilon,
\qquad |J|\leq N-1.
\end{align}
\end{lemma}
\begin{proof}
	{\bf Step 1.}  We first prove 
	\begin{align}  \label{eq:spacederiv}
		|\langle x \rangle^k \nabla^k \psi(t_0,x)| \lesssim \min \{\epsilon^{1/2} \langle r \rangle^{-1/2}, \langle r \rangle^{-1}  \},\ \ \forall\ 0\leq k\leq N-3. 
	\end{align}
Indeed, by applying the Klainerman-Sobolev inequality, Propositions \ref{prop:extradecay}--\ref{prop:commu2} on \eqref{eq:thm-2Ddata}, one can see
$$
| \nabla \nabla^k \psi(t_0)| \lesssim \min\{ \epsilon\langle r \rangle^{-1-k}, \langle r \rangle^{-2-k}\}  .
$$
Noticing that $ \nabla^k \psi(t_0, |x|=+\infty) =0$, then we have 
\begin{align}
	| \nabla^k \psi|(t_0, x) \leq  \int_{|x|}^{+\infty} \big|\partial_{r}  \nabla^k \psi(t_0,y) \big| d|y| \lesssim \epsilon^{1/2} \langle r \rangle^{-1/2-k}.  \label{eq:spacepoint}
\end{align}
Besides,  Propositions \ref{prop:extradecay}, \ref{prop:commu2} and Klainerman-Sobolev inequality indicate that
$$
|\nabla^k \psi(t_0)| \lesssim \langle r \rangle^{-1-k}.
$$ 
Thus we have \eqref{eq:spacederiv}.

{\bf Step 2.} We shall prove the following $L^2$  estimates
\begin{align}
	\|\langle x \rangle^{k}\partial \nabla^{k-1}  \psi(t_0,x)\| & \lesssim 1,  \ \ \qquad \forall \  1\leq k\leq N,   \label{eq:timepoint}  \\
	\|\langle x \rangle^{k}  \partial \nabla^k \psi(t_0)\| & \lesssim  \epsilon, \ \quad \forall \ 0\leq k \leq N-1,
\end{align}
and $L^{\infty}$ estimate
\begin{align*}
	|\langle x \rangle^k \partial \nabla^{k-1} \psi(t_0,x)| \lesssim  \min \{\epsilon^{1/2} \langle r \rangle^{-1/2}, \langle r \rangle^{-1}  \},\ \ \forall\ 1\leq k\leq N-3.
\end{align*}
Recall that we can rewrite equation \eqref{eq:Soler} as 
\begin{align} \label{eq:timeequation}
	\partial_t \psi = -\gamma^0 \gamma^a \partial_a \psi + i(\gamma^0 \psi^* H \psi)\psi.
\end{align}
For any $1\leq k \leq N$,  acting $\nabla^{k-1}$ to both sides of \eqref{eq:timeequation}, one can get 
\begin{align*}
	\|\langle x \rangle^{k}\partial_t \nabla^{k-1}\psi(t_0)\| \lesssim \|\langle x \rangle^{k}\nabla \nabla^{k-1}\psi\|+ \|\langle x \rangle^k\nabla^{k-1}[(\psi^*H \psi)\psi]\| \lesssim 1,
\end{align*}
where we use \eqref{eq:thm-2Ddata} and \eqref{eq:spacederiv} in the last step.  Furthermore, fix $0\leq k \leq N-1$, owing to \eqref{eq:thm-2Ddata} and \eqref{eq:spacederiv}, one can also obtain
\begin{align*}
	\|\langle x \rangle^{k}\partial_t \nabla^{k}\psi(t_0)\| \lesssim \|\langle x \rangle^{k}\nabla \nabla^{k}\psi\|+ \|\langle x \rangle^k\nabla^{k}[(\psi^*H \psi)\psi]\| \lesssim \epsilon.
\end{align*}
Finally, for $1\leq k\leq N-3$, we use \eqref{eq:timeequation}, \eqref{eq:spacederiv} to deduce that
\begin{align*}
	|\partial_t \nabla^{k-1}\psi(t_0,x)| & \lesssim |\nabla \nabla^{k-1}\psi(t_0,x)|+|\nabla^{k-1}[(\psi^*H\psi)\psi]| \\
	& \lesssim \min \{\epsilon^{1/2} \langle r \rangle^{-1/2-k}, \langle r \rangle^{-1-k}  \}.
\end{align*}

{\bf Step 3. Smallness and boundedness.} 
We prove 
\begin{align}
	\|\langle r \rangle^{|I|}\partial^{|I|}  \psi(t_0,x)\| & \lesssim 1,  \ \ \qquad \forall \ \   |I|\leq N,   \label{eq:timepoint}  \\
	\|\langle r \rangle^{|I|}  \partial \partial^{I}\psi(t_0)\| & \lesssim  \epsilon, \ \quad \forall \ 0\leq |I| \leq N-1, \notag \\
	\langle r \rangle^{|I|}|\partial^{I} \psi(t_0,x)| & \lesssim  \min \{\epsilon^{1/2} \langle r \rangle^{-1/2}, \langle r \rangle^{-1}  \},\ \ \forall\ 1\leq |I|\leq N-3. \notag
\end{align}

Indeed, one can repeat the procedure in Step 2 to replace $\nabla$ by $\partial$ one by one. An  induction argument will lead to the desired results.  The proof is done.
\end{proof}

Next, we set the bootstrap assumption. Since the energy of solution is continuous with respect to time,  one can assert that there exists some time $T>t_0$, such that for all $t\in [t_0, T)$, it holds
 \begin{equation}\label{eq:BA2d}
\aligned
E(\widehat{\Gamma}^I \psi, t )^{1/2} &\leq C_1,
\qquad |I|\leq N,
\\
 E(\widehat{\Gamma}^J \partial\psi, t )^{1/2} &\leq C_1 \epsilon,
\qquad |J|\leq N-1.
\endaligned
\end{equation}
	In the above $C_1 \gg 1$ is to be determined, and $\epsilon$ is sufficiently small such that $C_1^6 \epsilon \ll 1$. The maximal existence time $T^*$ is defined as
	\begin{equation}\label{eq:T-2d}
	T^* = \sup \{T: \eqref{eq:BA2d} \text{ holds} \},
	\end{equation}
and it holds $T^* > t_0$. Similar to the previous section, our goal is to verify $T^* = +\infty$, and then the solution $\psi$ exists globally. Below in this section we assume $T^*<+\infty$,  and then deduce a contradiction which implies $T^*$ must be $+\infty$. 	
With the help of commutator estimates, we can obtain the following $L^2$-type bounds.
\begin{proposition}	\label{prop:L2-2d}
	Let $\psi$ satisfy the estimates in \eqref{eq:BA2d}, then for all $t\in [t_0, T^*)$ one has
	$$
	\aligned
	&\sum_{|I|\leq N}\| \widehat{\Gamma}^I \psi \| 
	+ \sum_{|I|\leq N} \left( \int_{t_0}^t \Big\| {[\widehat{\Gamma}^I \psi]_- \over \langle s-r\rangle^{1/2+\delta}} \Big\|^2 \, ds \right)^{\frac 12}
	\leq C_1,
	\\
	&\sum_{|J|\leq N-1} \big(\| \widehat{\Gamma}^J \partial  \psi \| + \|  \partial \widehat{\Gamma}^J \psi \| \big)
	\lesssim C_1 \epsilon,
		\\
		 &\sum_{|J|\leq N-1} \left( \int_{t_0}^t \Big(\Big\| {[\widehat{\Gamma}^J \partial \psi]_- \over \langle s-r\rangle^{1/2+\delta}} \Big\|^2 
		+ \Big\| {[\partial \widehat{\Gamma}^J  \psi]_- \over \langle s-r\rangle^{1/2+\delta}} \Big\|^2 \Big) \, ds \right)^{\frac 12}
	\lesssim C_1 \epsilon.
	\endaligned
	$$
	\end{proposition}
	
	Combined with the Klainerman-Sobolev inequality in Proposition \ref{prop:K-S} and commutator estimates in Proposition \ref{prop:commu2}, we can derive the following pointwise estimates for $\psi$.
	\begin{proposition}\label{prop:pointwise-2d}
	Let $\psi$ satisfy the estimates in \eqref{eq:BA2d}, then for all $t\in [t_0, T^*)$ the following holds
	\begin{align}
	&\sum_{|I|\leq N-2}| \widehat{\Gamma}^I \psi | \lesssim C_1 \langle t+r\rangle^{-1/2} \langle t-r\rangle^{-1/2},  \label{eq:point-2d-1}
	\\
	&\sum_{|J|\leq N-3} | \widehat{\Gamma}^J  \partial \psi | + \sum_{|J|\leq N-3} | \partial \widehat{\Gamma}^J \psi | 
	\lesssim C_1 \epsilon \langle t+r\rangle^{-1/2} \langle t-r\rangle^{-1/2}.  \label{eq:point-2d-2}  
	\end{align}
	\end{proposition}

	\subsection{Improved estimates}	
	
	Our goal in this section is to prove improved bounds for the solution $\psi$, which is listed in the following proposition.
\begin{proposition}	\label{prop:improved-2d}
	If the estimates in \eqref{eq:BA2d} hold, for all $t\in [t_0, T^*)$ we have
	\begin{equation}
	\aligned
E(\widehat{\Gamma}^I \psi, t )^{1/2} &\leq {1\over 2} C_1,
\qquad |I|\leq N,
\\
 E(\widehat{\Gamma}^J \partial  \psi, t )^{1/2} &\leq {1\over 2} C_1 \epsilon,
\qquad |J|\leq N-1.
\endaligned
	\end{equation}
\end{proposition}

	To prove Proposition \ref{prop:improved-2d}, we need to prepare some bounds on  $\psi$. We note the pointwise estimates for $\widehat{\Gamma}^I \psi$ (with $|I|\leq N-2$) in Proposition \ref{prop:pointwise-2d} do not enjoy any smallness in $\epsilon$, which causes difficulty in showing Proposition \ref{prop:improved-2d}. The following result tells us that $\widehat{\Gamma}^J \psi$ (with $|J|\leq N-3$) can enjoy some smallness in $\epsilon$ at the expense of losing some decay rate.
	\begin{lemma} \label{lem:smdecaypsi}
	Let $0< \delta \ll {1/2} $ be a small parameter, then for all $t\in [t_0, T^*)$ we have
	\begin{equation}\label{eq:point-2d-3}
		\sum_{|I|\leq N-3}| \widehat{\Gamma}^I \psi | \lesssim  C_1 \epsilon^{1/2} \langle t+r\rangle^{-1/2+\delta}.
	\end{equation}
	\end{lemma}
	\begin{proof}
We assume $|I|\leq N-3$. 
Since the decay rate for $\psi$ is slower in 2D compared with 3D case, we need to modify the proof used in showing \eqref{eq:point-2d-3}.
Noting that $|\widehat{\Gamma}^{I} \psi|^2 (t, |x|=+\infty) = 0$ , by the fundamental theorem of calculus we know for all fixed $t\in [t_0, T)$ it holds that 
\begin{align*}
|\widehat{\Gamma}^{I} \psi|^2 (t, x)
&\lesssim
\int_{|x|}^{+\infty} |\widehat{\Gamma}^{I} \psi (t, y) | |\partial_r \widehat{\Gamma}^{I} \psi (t, y) | \, d|y|
\\
& \lesssim \int_{|x|}^{+\infty} |\widehat{\Gamma}^{I} \psi (t, y) | |\nabla \widehat{\Gamma}^{I} \psi (t, y) | \, d|y|.
\end{align*}

	Inserting \eqref{eq:point-2d-1} and \eqref{eq:point-2d-2} to this inequality we obtain
	$$
	\aligned
	|\widehat{\Gamma}^{I} \psi|^2 (t, x)
&\lesssim
C_1^2 \epsilon \langle t + r\rangle^{-1+2\delta} \int_{|x|}^{+\infty} \langle t-|y|\rangle^{-1-2\delta} \, d|y|
\\
&\lesssim
C_1^2 \epsilon \langle t+r\rangle^{-1+2\delta}.
\endaligned
	$$
	
The proof is done.
\end{proof}
	
	The following result indicates that $[\psi]_-$ decays faster than $\psi$ in time. We act $\widehat{\Gamma}^I$  to the equation \eqref{eq:Soler} and use \eqref{eq:vectorfield} to get
	\begin{align}
	-i\gamma^\mu \del_\mu \widehat{\Gamma}^I \psi &= \sum_{|L| \leq |I|} c_{L,I}\widehat{\Gamma}^L\big((\psi^* H \psi)  \psi\big), 
	\quad |I|\leq N. \label{eq:Dirac-2d-1}
	\end{align}
\begin{lemma}\label{lem:pointghost}
	Let $0< \delta \ll {1/2} $ be a small parameter. It holds for all $t\in [t_0, T^*)$ that
	\begin{align}
\sum_{|I|\leq N-3}\big|[\widehat{\Gamma}^{I} \psi]_-\big|	& \lesssim \min\big\{  C_1 \langle t+r\rangle^{-1+\delta}, C_1 \epsilon^{1/4} \langle t + r \rangle^{-3/4 + \delta} \big\},\label{eq:psi--point}
\\
\sum_{|J|\leq N-4}\big|[\widehat{\Gamma}^{J}  \partial \psi]_-\big|	& \lesssim C_1^3 \epsilon \langle t+r\rangle^{-1+\delta}. \label{eq:psi--point2} 
	\end{align}
	\end{lemma}
	\begin{proof}
It is easy to see that \eqref{eq:psi--point} and \eqref{eq:psi--point2} hold  in the region $\{ r\leq {t/2} \} \bigcup \{ r \geq 2t \}$.	 Thus below we will only consider the case of $t/2 \leq r \leq 2t$.

	\textbf{Step 1: Proof of \eqref{eq:psi--point}.} Fix $|I| \leq N-3$,
by Lemma \ref{lem:D-Point-improve}, we know for $t/2 \leq r \leq 2t$ it holds
	\begin{align*}
	\big|\partial_r [\widehat{\Gamma}^{I} \psi]_-\big|
	&\lesssim
	\langle t+r\rangle^{-1} \big( \big|\widehat{\Gamma} \widehat{\Gamma}^I \psi \big| + \big| \widehat{\Gamma}^I \psi \big| \big)
	+ \sum_{|L|\leq |I|}\big|\widehat{\Gamma}^L \big((\psi^* H \psi)  \psi\big) \big|
	\\
	&\lesssim
	(C_1 + C_1^3 \epsilon^{1/2}) \langle t+r\rangle^{-1+\delta} \langle t-r\rangle^{-1-\delta}. 
	\end{align*}
	Next, we apply  the fundamental theorem of calculus for all fixed $t\in [t_0, T)$ and the fact $|[\widehat{\Gamma}^{I} \psi]_-(t, |x|=2t)|\lesssim C_1 \langle t+|x|\rangle^{-1}$ to get
	\begin{align*}
	\big|[\widehat{\Gamma}^{I} \psi]_-\big|
	&\lesssim
		\big|[\widehat{\Gamma}^{I} \psi]_-\big|(t, |x|=2t)
	+\int_{|x|}^{+\infty} 	\big|\partial_r [\widehat{\Gamma}^{I} \psi]_-\big|(t, y) \,  d|y|
	\\
&\lesssim  C_1 \langle t+|x|\rangle^{-1}
+(C_1 + C_1^3 \epsilon^{1/2}) \int_{|x|}^{+\infty} \langle t+|y|\rangle^{-1+\delta} \langle t-|y|\rangle^{-1-\delta} \, d|y|
\\
	&\lesssim (C_1 + C_1^3 \epsilon^{1/2}) \langle t+|x|\rangle^{-1+\delta}.
	\end{align*}
By the fact $C_1^6 \epsilon \ll 1$, we arrive at 
\begin{align} \label{eq:ghostde2d1}
	\big|[\widehat{\Gamma}^{I} \psi]_-\big|
	&\lesssim C_1\langle t+|x|\rangle^{-1+\delta}.
\end{align}
Furthermore, using  Lemma \ref{lem:smdecaypsi}, we also have 
\begin{align} \label{eq:ghostde2d2}
	\big|[\widehat{\Gamma}^{I} \psi]_-\big|
	&\lesssim |\widehat{\Gamma}^{I} \psi| \lesssim C_1 \epsilon^{1/2} \langle t+r\rangle^{-1/2+\delta}.
\end{align} 
Taking interpolation between \eqref{eq:ghostde2d1} and \eqref{eq:ghostde2d2}, one can deduce that 
\begin{align*}
	\big|[\widehat{\Gamma}^{I} \psi]_-\big|
	&\lesssim C_1 \epsilon^{1/4} \langle t + r \rangle^{-3/4 + \delta}.
\end{align*}
So \eqref{eq:psi--point} is proved.

	\textbf{Step 2: Proof of \eqref{eq:psi--point2}.} Fix $|J| \leq N-4.$
	By Lemma \ref{lem:D-Point-improve}, we know for $t/2 \leq r \leq 2t$ it holds
	\begin{align*}
	\big|\partial_r [\widehat{\Gamma}^{J} \partial \psi]_-\big|
	&\lesssim
	\langle t+r\rangle^{-1} \big( \big|\widehat{\Gamma} \widehat{\Gamma}^J \partial \psi \big| + \big| \widehat{\Gamma}^J \partial \psi \big| \big)
	+ \big|\widehat{\Gamma}^J \partial \big((\psi^* H \psi)  \psi\big) \big|
	\\
	&\lesssim
	C_1^3 \epsilon \langle t+r\rangle^{-1+\delta} \langle t-r\rangle^{-1-\delta}. 
	\end{align*}
	Next, we apply  the fundamental theorem of calculus for all fixed $t\in [t_0, T)$ and the fact $|[\widehat{\Gamma}^{J} \partial \psi]_-(t, |x|=2t)|\lesssim C_1 \epsilon \langle t+|x|\rangle^{-1+\delta}$(see \eqref{eq:point-2d-2}) to get
	\begin{align*}
	\big|[\widehat{\Gamma}^{J} \partial \psi]_-\big|
	&\lesssim
	|[\widehat{\Gamma}^{J} \partial \psi]_-(t, |x|=2t)|
+	\int_{|x|}^{+\infty} 	\big|\partial_r [\widehat{\Gamma}^{J} \partial \psi]_-\big|(t, y) \,  d|y|
	\\
&\lesssim C_1 \epsilon \langle t+|x|\rangle^{-1 + \delta}
+C_1^3 \epsilon \int_{|x|}^{+\infty} \langle t+|y|\rangle^{-1+\delta} \langle t-|y|\rangle^{-1-\delta} \, d|y|
\\
	&\lesssim C_1^3 \epsilon \langle t+|x|\rangle^{-1+\delta}.
	\end{align*}
	
	The proof is completed.
	\end{proof}	
	Next, we use energy estimates to improve the bounds in \eqref{eq:BA2d}. We act  $\widehat{\Gamma}^I$ and  $\widehat{\Gamma}^J \partial $ respectively to the equation \eqref{eq:Soler} to get
	\begin{align}
	-i\gamma^\mu \del_\mu \widehat{\Gamma}^I  \psi &= \sum_{|L|\leq |I|} c_{L,I}\widehat{\Gamma}^L  \big((\psi^* H \psi) \psi\big),  \label{eq:Dirac-2dderi}
\quad |I|\leq N. \\
-i\gamma^\mu \del_\mu \widehat{\Gamma}^J \partial \psi &= \sum_{|K|\leq |J|}c_{K,J}\widehat{\Gamma}^K \partial \big((\psi^* H \psi) \psi\big), 
\quad |J|\leq N-1. \label{eq:Dirac-2d-2}
	\end{align}
	
	\begin{proposition}	\label{prop:improved-2d-rough}
	Let $N\geq 7$, if the estimates in \eqref{eq:BA2d} hold, for all $t\in [t_0, T^*)$ we have
	\begin{align}
	E(\widehat{\Gamma}^I \psi, t ) &\lesssim  1 + C_1^4 \epsilon^{1/2},
\qquad |I|\leq N,   \label{eq:2d-rough1}
\\
 E(\widehat{\Gamma}^J \partial \psi, t ) &\lesssim \epsilon^2 + C_1^{6} \epsilon^{9/4},
\qquad |J|\leq N-1.  \label{eq:2d-rough2}
	\end{align}

\end{proposition}	
	\begin{proof}
	\textbf{Step 1: Proof of \eqref{eq:2d-rough1}.}
	
	Let $N\geq 7$, for any multi-index $|I| \leq N$, we apply the energy estimates  in Proposition \ref{prop:DiracEE} on \eqref{eq:Dirac-2dderi} to get for all $t\in [t_0, T^*)$ that
	\begin{align*}
	E(\widehat{\Gamma}^I \psi, t )
	&\lesssim E(\widehat{\Gamma}^I \psi, t_0 ) 
	+\sum_{|L|\leq |I|} \int_{t_0}^t \big\| (\widehat{\Gamma}^I \psi)^* \gamma^0  \widehat{\Gamma}^L \big((\psi^* H \psi)  \psi\big)  \big\|_{L^1} \, ds.
	\end{align*}
	Next, we focus on the bound of $\sum_{|L|\leq |I|}\big\| (\widehat{\Gamma}^I \psi)^* \gamma^0  \widehat{\Gamma}^L \big((\psi^* H \psi) \psi\big)  \big\|_{L^1}$. By Lemma \ref{lem:preserve}, we have
	\begin{align*}
	&\sum_{|L|\leq |I|} \big\| (\widehat{\Gamma}^I \psi)^* \gamma^0  \widehat{\Gamma}^L \big((\psi^* H \psi) \psi\big)  \big\|_{L^1} 
	\\
	\lesssim
	&\sum_{|I_1|+|I_2|\leq |I|} \big\| \big|(\widehat{\Gamma}^I \psi)^* \gamma^0 \widehat{\Gamma}^{I_1} \psi\big| \, \big|{\Gamma}^{I_2} (\psi^* H \psi)\big|  \big\|_{L^1} 
	\\
	\lesssim
	&\sum_{|I_1|\leq N-3,\, |I_2|\leq N} \big\| \big|(\widehat{\Gamma}^I \psi)^* \gamma^0 \widehat{\Gamma}^{I_1} \psi\big| \, \big|{\Gamma}^{I_2} (\psi^* H \psi)\big|  \big\|_{L^1} 
	\\
	+
	&\sum_{|I_1|\leq N, \, |I_2|\leq N-3} \big\| \big|(\widehat{\Gamma}^I \psi)^* \gamma^0 \widehat{\Gamma}^{I_1} \psi\big| \, \big|{\Gamma}^{I_2} (\psi^* H \psi)\big|  \big\|_{L^1} 
=: B_1 + B_2.
	\end{align*}
	Applying Lemma \ref{lem:decompose}, one can see
	\begin{align*}
	B_1
	\lesssim
	&\sum_{|I_1|\leq N-3,\, |I_2|\leq N} \big\| \big|[\widehat{\Gamma}^I \psi]_- \big| \, \big|\widehat{\Gamma}^{I_1} \psi\big| \, \big|{\Gamma}^{I_2} (\psi^* H \psi)\big|  \big\|_{L^1} 
	\\
	+
	&\sum_{|I_1|\leq N-3,\, |I_2|\leq N} \big\| \big|\widehat{\Gamma}^I \psi \big| \, \big|[\widehat{\Gamma}^{I_1} \psi]_-\big| \, \big|{\Gamma}^{I_2} (\psi^* H \psi)\big|  \big\|_{L^1} 
	=: B_{11} + B_{12}.
	\end{align*}
	To bound $B_{11}$, we apply  the $L^2$ estimates in Proposition \ref{prop:L2-2d} and pointwise estimates in \eqref{eq:point-2d-1} and \eqref{eq:point-2d-3} to get
	\begin{align*}
	B_{11}
	\lesssim
	&\sum_{\substack{|I_1|\leq N-3,\, |I_2^1|\leq N \\ |I_2^2| \leq N-3}} \bigg\| {[\widehat{\Gamma}^I \psi]_- \over \langle s-r\rangle^{1/2+\delta}}\bigg\| \, \big\|\widehat{\Gamma}^{I_2^1} \psi\big\| \, \big\| \langle s-r\rangle^{1/2+\delta} \widehat{\Gamma}^{I_1} \psi  \big\|_{L^\infty}\|\widehat{\Gamma}^{I_2^2}\psi\|_{L^{\infty}} 
		\\
		\lesssim & C_1^3 \epsilon^{1/2} \langle s\rangle^{-1+2\delta} \bigg\| {[\widehat{\Gamma}^I \psi]_- \over \langle s-r\rangle^{1/2+\delta}}\bigg\|.
	\end{align*}
	For $B_{12}$, one can see 
	\begin{align*}
	B_{12}
	\lesssim
	&\sum_{\substack{|I_1|\leq N-3,\, |I^1_2|\leq N \\ |I_2^2| \leq N-3}} \big\|\widehat{\Gamma}^I \psi \big\| \, \big\|[\widehat{\Gamma}^{I_1} \psi]_-\big\|_{L^\infty} \, \big\|\widehat{\Gamma}^{I_2^1} \psi \big\| \, \big\|\widehat{\Gamma}^{I_2^2} \psi  \big\|_{L^\infty} 
	\\
	\lesssim
	& C_1^4 \epsilon^{1/2} \langle s\rangle^{-3/2+2\delta},
	\end{align*}
		in which we use the $L^2$ estimates in Proposition \ref{prop:L2-2d} and pointwise estimates in  \eqref{eq:point-2d-3} and \eqref{eq:psi--point}.

Again applying Lemma \ref{lem:decompose}, one gets
	\begin{align*}
	B_2
	\lesssim
	&\sum_{|I_1|\leq N, \, |I_2|\leq N-3} \big\| \big|[\widehat{\Gamma}^I \psi]_- \big| \,  \big|\widehat{\Gamma}^{I_1} \psi\big| \, \big|{\Gamma}^{I_2} (\psi^* H \psi)\big|  \big\|_{L^1} 
	\\
	+
	&\sum_{|I_1|\leq N, \, |I_2|\leq N-3} \big\| \big|\widehat{\Gamma}^I \psi \big| \,  \big|[\widehat{\Gamma}^{I_1} \psi]_-\big| \, \big|{\Gamma}^{I_2} (\psi^* H \psi)\big|  \big\|_{L^1} 
	=: B_{21} + B_{22}.
	\end{align*}	
	For $B_{21}$, we have
	\begin{align*}
	B_{21}
	\lesssim
	&\sum_{\substack{|I_1|\leq N, \, |I_2^1|\leq N-3\\|I_2^2| \leq N-3}} \bigg\| {[\widehat{\Gamma}^I \psi]_- \over \langle s-r\rangle^{1/2+\delta}}\bigg\| \,  \big\|\widehat{\Gamma}^{I_1} \psi\big\| \, \big\|\langle s-r\rangle^{1/2+\delta} \widehat{\Gamma}^{I_2^1} \psi  \big\|_{L^\infty}  \big\| \widehat{\Gamma}^{I_2^2} \psi  \big\|_{L^\infty} 
	\\
	\lesssim
	 & C_1^3 \epsilon^{1/2} \langle s\rangle^{-1+2\delta} \bigg\| {[\widehat{\Gamma}^I \psi]_- \over \langle s-r\rangle^{1/2+\delta}}\bigg\|,
	\end{align*}
	in which we use the $L^2$ estimates in Proposition \ref{prop:L2-2d} and pointwise estimates in \eqref{eq:point-2d-1} and \eqref{eq:point-2d-3}.
	Concerning the estimate of $B_{22}$, we derive
	\begin{align*}
	B_{22}
	\lesssim
	&\sum_{\substack{|I_1|\leq N, \, |I_2^1|\leq N-3\\|I_2^2| \leq N-3}}  \big\| \widehat{\Gamma}^I \psi \big\| \,  \bigg\|{[\widehat{\Gamma}^{I_1} \psi]_{-}\over \langle s-r\rangle^{1/2+\delta}}\bigg\| \, \big\|\langle s-r\rangle^{1/2+\delta} \widehat{\Gamma}^{I_2^1} \psi  \big\|_{L^\infty}\big\| \widehat{\Gamma}^{I_2^2} \psi  \big\|_{L^\infty} 
	\\
	\lesssim
	  &C_1^3 \epsilon^{1/2} \langle s\rangle^{-1+2\delta} \sum_{|I_1|\leq N} \bigg\| {[\widehat{\Gamma}^{I_1} \psi]_- \over \langle s-r\rangle^{1/2+\delta}}\bigg\|,
	\end{align*}
		in which we use  the estimates in Proposition \ref{prop:L2-2d}, \eqref{eq:point-2d-1}, and \eqref{eq:point-2d-3}.

	To sum things up, we have
	\begin{align*}
		\sum_{|L|\leq |I|} \big\| (\widehat{\Gamma}^I \psi)^* \gamma^0  \widehat{\Gamma}^L \big((\psi^* H \psi)  \psi\big)  \big\|_{L^1} 
\lesssim
& C_1^3 \epsilon^{1/2} \langle s\rangle^{-1+2\delta} \sum_{|I_1|\leq N} \bigg\| {[\widehat{\Gamma}^{I_1} \psi]_- \over \langle s-r\rangle^{1/2+\delta}}\bigg\| \\ & +  C_1^4 \epsilon^{1/2} \langle s\rangle^{-3/2+2\delta}.
	\end{align*}
	Thus, we get
	\begin{align*}
		E(\widehat{\Gamma}^I \psi, t )
\lesssim
&1 + C_1^4 \epsilon^{1/2}  +C_1^3 \epsilon^{1/2} \sum_{|I_1|\leq N} \int_{t_0}^t  \langle s\rangle^{-1+2\delta}  \bigg\| {[\widehat{\Gamma}^{I_1} \psi]_- \over \langle s-r\rangle^{1/2+\delta}}\bigg\| \, ds
\\
\lesssim
&1 + C_1^4 \epsilon^{1/2}+ C_1^3 \epsilon^{1/2} \sum_{|I_1|\leq N} \Big( \int_{t_0}^t \bigg\| {[\widehat{\Gamma}^{I_1} \psi]_- \over \langle s-r\rangle^{1/2+\delta}}\bigg\|^2 \, ds\Big)^{1/2}  
 \\
 	\lesssim
 &	1 + C_1^4 \epsilon^{1/2}.
	\end{align*}
	This yields \eqref{eq:2d-rough1}.
 
		\textbf{Step 2: Proof of \eqref{eq:2d-rough2}.}

	Fix $N \geq 7$ and $|J|\leq N-1$. To show the second estimate in Proposition \ref{prop:improved-2d-rough}, we apply Proposition \ref{prop:DiracEE} on \eqref{eq:Dirac-2d-2} to derive for all $t\in [t_0, T^*)$ that
	\begin{align}
	E(\widehat{\Gamma}^J \partial \psi, t )
	&\lesssim E(\widehat{\Gamma}^J \partial \psi, t_0 ) 
	+\sum_{|K|\leq |J|}  \int_{t_0}^t \big\| \widehat{\Gamma}^J \partial \psi^* \gamma^0  \widehat{\Gamma}^K \partial \big((\psi^* H \psi)  \psi\big)  \big\|_{L^1} \, ds \nonumber
	\\
	&\lesssim
	\epsilon^2 +\sum_{|J_1|+ |J_2|\leq |J|}   \int_{t_0}^t \big\| \big(\widehat{\Gamma}^J \partial \psi^* \gamma^0 \widehat{\Gamma}^{J_2} \psi \big) {\Gamma}^{J_1} \partial (\psi^* H \psi)   \big\|_{L^1} \, ds \nonumber\\
	&  + \sum_{|J_1|+ |J_2|\leq |J|}   \int_{t_0}^t \big\| \big(\widehat{\Gamma}^J \partial \psi^* \gamma^0  \widehat{\Gamma}^{J_2}\partial \psi \big) {\Gamma}^{J_1} (\psi^* H \psi)    \big\|_{L^1} \, ds. \label{eq:derivaener}
	\end{align}
	One can observe that 
	\begin{align*}
		& \sum_{|J_1|+ |J_2|\leq |J|}   \big\| \big(\widehat{\Gamma}^J \partial \psi^* \gamma^0 \widehat{\Gamma}^{J_2} \psi \big) {\Gamma}^{J_1} \partial (\psi^* H \psi)   \big\|_{L^1}    \\
		& \lesssim \sum_{|J_2|\leq N-1, |J_1| \leq N-3 }    \big\| \big(\widehat{\Gamma}^J \partial \psi^* \gamma^0 \widehat{\Gamma}^{J_2} \psi \big) {\Gamma}^{J_1} \partial (\psi^* H \psi)   \big\|_{L^1}    \\
		& + \sum_{|J_1|\leq N-1, |J_2| \leq N-3 }   \big\| \big(\widehat{\Gamma}^J \partial \psi^* \gamma^0 \widehat{\Gamma}^{J_2} \psi \big) {\Gamma}^{J_1} \partial (\psi^* H \psi)   \big\|_{L^1}   := R_1+ R_2.
	\end{align*}
	Let us first treat $R_1$. We have 
	\begin{align*}
		R_1  & \lesssim  \sum_{ \substack{|J_2|\leq N-1, |J_1^1| \leq N- 3\\ |J_1^2| \leq N-3 } }   \big\| |[\widehat{\Gamma}^J \partial \psi]_{-} |\,|\widehat{\Gamma}^{J_2}\psi| \, |\widehat{\Gamma}^{J_1^1} \partial \psi| \, |\widehat{\Gamma}^{J_1^2} \psi|   \big\|_{L^1}    \\
		 & +  \sum_{ \substack{|J_2|\leq N-1, |J_1^1| \leq N- 3\\ |J_1^2| \leq N- 3} }   \big\| |\widehat{\Gamma}^J \partial \psi |\, |[\widehat{\Gamma}^{J_2}\psi]_{-}| \, |\widehat{\Gamma}^{J_1^1} \partial \psi| \,|\widehat{\Gamma}^{J_1^2} \psi|   \big\|_{L^1}  : = R_1 ^1+ R^2_1.
	\end{align*}
	Then, 
	\begin{align}
		R_1^1 & \lesssim \sum_{ \substack{|J_2|\leq N-1, |J_1| \leq N- 3 } }  \bigg \| \frac{[\widehat{\Gamma}^J \partial \psi]_{-} }{\langle s-r\rangle^{\frac 12 +\delta}} \bigg \| \big \| \widehat{\Gamma}^{J_2}\psi \big \| \big \| \langle s-r\rangle^{\frac 12 +\delta} \widehat{\Gamma}^{J_1} \partial \psi  \big\|_{L^{\infty}} \|\widehat{\Gamma}^{J_1} \psi\|_{L^{\infty}} \nonumber \\
		& \lesssim C_1^3 \epsilon^{\frac 32} \langle s \rangle^{-1+2\delta}\bigg \| \frac{[\widehat{\Gamma}^J \partial \psi]_{-} }{\langle s-r\rangle^{\frac 12 +\delta}} \bigg \|, 
	\end{align}
	in which Propositions \ref{prop:L2-2d}, \ref{prop:pointwise-2d} and Lemma \ref{lem:smdecaypsi} are used.  Regarding the estimate of $R_1^2$, using  Propositions \ref{prop:L2-2d}, \ref{prop:pointwise-2d} and Lemma \ref{lem:smdecaypsi} once again, we have
	\begin{align*}
		R^2_1 & \lesssim \sum_{ |J_2|\leq N-1, |J_1| \leq N- 3 }   \big\| \widehat{\Gamma}^J \partial \psi \big\| \bigg\|\frac{ [\widehat{\Gamma}^{J_2}\psi]_{-}}{\langle s-r\rangle^{\frac 12 +\delta}} \bigg\|  \big\|\langle s-r\rangle^{\frac 12 +\delta}\widehat{\Gamma}^{J_1} \partial \psi \big\|_{L^{\infty}} \big\|\widehat{\Gamma}^{J_1} \psi   \big\|_{L^{\infty}}  \\
		 & \lesssim C_1^3 \epsilon^{\frac 52} \langle s \rangle^{-1+2\delta} \sum_{|J_2|\leq N-1}\bigg\|\frac{ [\widehat{\Gamma}^{J_2}\psi]_{-}}{\langle s-r\rangle^{\frac 12 +\delta}} \bigg\|.
	\end{align*}
	Hence, collecting the estimates of $R_1^1$ and $R^2_1$, we get 
	\begin{align*}
		R_1 \lesssim C_1^3 \epsilon^{\frac 32} \langle s \rangle^{-1+2\delta}\bigg \| \frac{[\widehat{\Gamma}^J \partial \psi]_{-} }{\langle s-r\rangle^{\frac 12 +\delta}} \bigg \|+ C_1^3 \epsilon^{\frac 52} \langle s \rangle^{-1+2\delta} \sum_{|J_2|\leq N-1}\bigg\|\frac{ [\widehat{\Gamma}^{J_2}\psi]_{-}}{\langle s-r\rangle^{\frac 12 +\delta}} \bigg\|.
	\end{align*}
	On the other hand,  one can have 
	\begin{align*}
		R_2 & \lesssim \sum_{\substack{|J_1^1|\leq N-1, |J_2| \leq N-3 \\ |J_1^2| \leq N-3}}   \big\| \big|[\widehat{\Gamma}^J \partial \psi]_{-} \big| \big| \widehat{\Gamma}^{J_2}\psi\big| \big| \widehat{\Gamma}^{J_1^1} \partial \psi\big| \big| \widehat{\Gamma}^{J_1^2}\psi\big|   \big\|_{L^1} \\
		& + \sum_{\substack{|J_1^2|\leq N-1, |J_2| \leq N-3 \\ |J_1^1| \leq N-3}}   \big\| \big|[\widehat{\Gamma}^J \partial \psi]_{-} \big| \big| \widehat{\Gamma}^{J_2}\psi\big| \big| \widehat{\Gamma}^{J_1^1} \partial \psi\big| \big| \widehat{\Gamma}^{J_1^2}\psi\big|   \big\|_{L^1}  \\
		& + \sum_{\substack{|J_1^1|\leq N-1, |J_2| \leq N-3 \\ |J_1^2| \leq N-3}}   \big\| \big|\widehat{\Gamma}^J \partial \psi \big| \big| [ \widehat{\Gamma}^{J_2}\psi]_{-}\big| \big| \widehat{\Gamma}^{J_1^1} \partial \psi\big| \big| \widehat{\Gamma}^{J_1^2}\psi\big|   \big\|_{L^1} \\
		& +  \sum_{\substack{|J_1^2|\leq N-1, |J_2| \leq N-3 \\ |J_1^1| \leq N-3}}   \big\| \big|\widehat{\Gamma}^J \partial \psi \big| \big| [ \widehat{\Gamma}^{J_2}\psi]_{-}\big| \big| \widehat{\Gamma}^{J_1^1} \partial \psi\big| \big| \widehat{\Gamma}^{J_1^2}\psi\big|   \big\|_{L^1}.
	\end{align*}
	Owing to H\"older inequality, Propositions \ref{prop:L2-2d}, \ref{prop:pointwise-2d} and  Lemmas \ref{lem:smdecaypsi}, \ref{lem:pointghost},  we can further bound $R_2$ as 
	\begin{align*}
		R_2 & \lesssim C_1^3 \epsilon^{\frac 32} \langle s\rangle^{-1+2\delta} \bigg\|\frac{ [\widehat{\Gamma}^{J}\partial \psi]_{-}}{\langle s-r\rangle^{\frac 12 +\delta}} \bigg\| + C_1^4 \epsilon^{\frac 52} \langle s\rangle^{-\frac 32 + 2 \delta} + C_1^4 \epsilon^{\frac 94} \langle s \rangle^{-\frac 54 + \delta}  \\
		&  \lesssim C_1^3 \epsilon^{\frac 32} \langle s\rangle^{-1+2\delta} \bigg\|\frac{ [\widehat{\Gamma}^{J}\partial \psi]_{-}}{\langle s-r\rangle^{\frac 12 +\delta}} \bigg\| + C_1^4 \epsilon^{\frac 94} \langle s \rangle^{-\frac 54 + \delta}.
	\end{align*}
	Next, we estimate the other term in \eqref{eq:derivaener}
	\begin{align*}
	 & \sum_{|J_1|+ |J_2|\leq |J|}   \big\| \big(\widehat{\Gamma}^J \partial \psi^* \gamma^0 \widehat{\Gamma}^{J_2}\partial \psi \big) {\Gamma}^{J_1} (\psi^* H \psi)   \big\|_{L^1} 	 \\
	 & \lesssim  \sum_{|J_2|\leq N-1, |J_1|\leq N-4}   \big\| \big(\widehat{\Gamma}^J \partial \psi^* \gamma^0 \widehat{\Gamma}^{J_2}\partial \psi \big) {\Gamma}^{J_1} (\psi^* H \psi)   \big\|_{L^1} \\
	 & \qquad\qquad  +  \sum_{|J_1|\leq N-1, |J_2|\leq N-4}   \big\| \big(\widehat{\Gamma}^J \partial \psi^* \gamma^0 \widehat{\Gamma}^{J_2}\partial \psi \big) {\Gamma}^{J_1} (\psi^* H \psi)   \big\|_{L^1} := Q_1 + Q_2.
	\end{align*}
   In order to bound $Q_1$, we proceed to have
   \begin{align*}
   	Q_1 & \lesssim \sum_{\substack{|J_2|\leq N-1, |J_1^1|\leq N-4\\ |J_1^2| \leq N-4}}   \big\| \big|\big[\widehat{\Gamma}^J \partial \psi \big]_{-} \big| |\widehat{\Gamma}^{J_2}\partial \psi| | \widehat{\Gamma}^{J_1^1} \psi| | \widehat{\Gamma}^{J_1^2}\psi|   \big\|_{L^1} \\
   	  & \qquad  +  \sum_{\substack{|J_2|\leq N-1, |J_1^1|\leq N-4\\ |J_1^2| \leq N-4}}   \big\| \big|\widehat{\Gamma}^J \partial \psi  \big| \big|\big[\widehat{\Gamma}^{J_2}\partial \psi\big]_{-}\big| | \widehat{\Gamma}^{J_1^1} \psi| | \widehat{\Gamma}^{J_1^2}\psi|   \big\|_{L^1} : = Q_1^1 + Q_1^2.
   	     \end{align*}
   	     By virtue of Propositions \ref{prop:L2-2d}, \ref{prop:pointwise-2d}, and Lemma \ref{lem:smdecaypsi}, one can see 
   	     \begin{align*}
   	     	Q_1^1 & \lesssim \sum_{\substack{|J_2|\leq N-1, |J_1^1|\leq N-4\\ |J_1^2| \leq N-4}}   \bigg \|\frac{\big[\widehat{\Gamma}^J \partial \psi \big]_{-}}{\langle s -r \rangle^{\frac 12 + \delta}}  \bigg\| \big\|\widehat{\Gamma}^{J_2}\partial \psi \big\|\big\|\langle s -r \rangle^{\frac 12 + \delta}  \widehat{\Gamma}^{J_1^1} \psi\big\|_{L^{\infty}} \|\widehat{\Gamma}^{J_1^2} \psi\big\|_{L^{\infty}} \\
   	     	 & \lesssim C_1^3 \epsilon^{\frac 32} \langle s \rangle^{-1+2\delta}\bigg \| \frac{[\widehat{\Gamma}^J \partial \psi]_{-} }{\langle s-r\rangle^{\frac 12 +\delta}} \bigg \|.
   	     \end{align*}
   	     Similarly, we bound $Q_1^2$ as 
   	     \begin{align*}
   	     	Q_1^2 \lesssim C_1^3 \epsilon^{\frac 32} \langle s \rangle^{-1+2\delta}\sum_{|J_2|\leq N-1} \bigg \| \frac{[\widehat{\Gamma}^{J_2} \partial \psi]_{-} }{\langle s-r\rangle^{\frac 12 +\delta}} \bigg \|.
   	     \end{align*}
   	     Gathering the estimates of $Q_1^1$ and $Q_1^2$, one can find 
   	     \begin{align*}
   	     	Q_1 \lesssim  C_1^3 \epsilon^{\frac 32} \langle s \rangle^{-1+2\delta}\sum_{|J_2|\leq N-1} \bigg \| \frac{[\widehat{\Gamma}^{J_2} \partial \psi]_{-} }{\langle s-r\rangle^{\frac 12 +\delta}} \bigg \|.
   	     \end{align*}
   	     Now we consider the estimate of $Q_2$. We have 
   	     \begin{align*}
   	     	Q_2  & \lesssim  \sum_{\substack{|J_1^1|\leq N-1, |J_2|\leq N-4\\ |J_1^2| \leq N-4}}   \big\| \big|\big[\widehat{\Gamma}^J \partial \psi \big]_{-}\big| \big|\widehat{\Gamma}^{J_2}\partial \psi \big| \big|\widehat{\Gamma}^{J_1^1} \psi \big| \big| \widehat{\Gamma}^{J_1^2}\psi \big|  \big\|_{L^1} \\
   	     	&\qquad +  \sum_{\substack{|J_1^1|\leq N-1, |J_2|\leq N-4\\ |J_1^2| \leq N-4}}   \big\| \big|\widehat{\Gamma}^J \partial \psi \big| \big|\big[\widehat{\Gamma}^{J_2}\partial \psi \big]_{-}\big| \big|\widehat{\Gamma}^{J_1^1} \psi \big| \big| \widehat{\Gamma}^{J_1^2}\psi \big|  \big\|_{L^1}:= Q_2^1+Q_2^2.
   	     \end{align*}
   	     Concerning the bound of $Q_2^1$, one can see 
   	     \begin{align*}
   	     	Q_2^1 & \lesssim  \sum_{\substack{|J_1^1|\leq N-1, |J_2|\leq N-4\\ |J_1^2| \leq N-4}}   \bigg\| \frac{\big[\widehat{\Gamma}^J \partial \psi \big]_{-} }{\langle s-r\rangle^{\frac 12 + \delta}}\bigg\|  \big\| \langle s-r\rangle^{\frac 12 + \delta}\widehat{\Gamma}^{J_2}\partial \psi \big\|_{L^{\infty}} \big\|{\widehat\Gamma}^{J_1^1} \psi \big\| \big\| {\widehat\Gamma}^{J_1^2}\psi   \big\|_{L^{\infty}} \\
   	     	  & \lesssim C_1^3 \epsilon^{\frac 32} \langle s \rangle^{-1+2\delta}\bigg \| \frac{[\widehat{\Gamma}^J \partial \psi]_{-} }{\langle s-r\rangle^{\frac 12 +\delta}} \bigg \|,
   	     \end{align*}
   	    in which Propositions \ref{prop:L2-2d}, \ref{prop:pointwise-2d} are used.  Regarding $Q_2^2$,  one can have
   	       	    \begin{align*}
   	    	Q_2^2 & \lesssim \sum_{\substack{|J_1^1|\leq N-1, |J_2|\leq N-4\\ |J_1^2| \leq N-4}}   \big\| \widehat{\Gamma}^J \partial \psi  \big\| \big\|\big[\widehat{\Gamma}^{J_2}\partial \psi \big]_{-}\big\|_{L^{\infty}} \big\|{\Gamma}^{J_1^1} \psi \big\| \big\| {\Gamma}^{J_1^2}\psi   \big\|_{L^{\infty}}  \\
   	    	    & \lesssim C_1^6 \epsilon^{\frac 52} \langle s \rangle^{-\frac 32 + 2\delta},
   	    \end{align*} 
   	    where Proposition \ref{prop:L2-2d} and Lemmas \ref{lem:smdecaypsi}, \ref{lem:pointghost} are used.
   	     Following from the bounds of $Q_2^1$ and $Q_2^2$, we see 
   	     \begin{align*}
   	     	Q_2 \lesssim C_1^3 \epsilon^{\frac 32} \langle s \rangle^{-1+2\delta}\bigg \| \frac{[\widehat{\Gamma}^J \partial \psi]_{-} }{\langle s-r\rangle^{\frac 12 +\delta}} \bigg \| + C_1^6 \epsilon^{\frac 52} \langle s \rangle^{-\frac 32 + 2\delta}.
   	     \end{align*}
   	     Now collecting  the estimates of $R_1$ and $R_2$, we obtain 
   	     \begin{align}
   	     &  \sum_{|J_1|+ |J_2|\leq |J|}   \big\| \big(\widehat{\Gamma}^J \partial \psi^* \gamma^0 \widehat{\Gamma}^{J_2} \psi \big) {\Gamma}^{J_1} \partial (\psi^* H \psi)   \big\|_{L^1}  \nonumber \\
   	     	&  \lesssim C_1^3 \epsilon^{\frac 32} \langle s \rangle^{-1+2\delta}\bigg \| \frac{[\widehat{\Gamma}^J \partial \psi]_{-} }{\langle s-r\rangle^{\frac 12 +\delta}} \bigg \| + C_1^4 \epsilon^{\frac 94} \langle s \rangle^{-\frac 54 + \delta} \nonumber\\
   	     	& \qquad \qquad  + C_1^3 \epsilon^{\frac 52} \langle s \rangle^{-1+2\delta} \sum_{|J_2|\leq N-1}\bigg\|\frac{ [\widehat{\Gamma}^{J_2}\psi]_{-}}{\langle s-r\rangle^{\frac 12 +\delta}} \bigg\|  \label{eq:mediuesti1}.
   	     \end{align}
   	     Gathering the bounds of $Q_1$ and $Q_2$, one can get 
   	     \begin{align}
   	     	& \sum_{|J_1|+ |J_2|\leq |J|}   \big\| \big(\widehat{\Gamma}^J \partial \psi^* \gamma^0 \widehat{\Gamma}^{J_2}\partial \psi \big) {\Gamma}^{J_1} (\psi^* H \psi)   \big\|_{L^1} 	\nonumber \\
   	     	& \lesssim  C_1^3 \epsilon^{\frac 32} \langle s \rangle^{-1+2\delta}\sum_{|J_2|\leq N-1} \bigg \| \frac{[\widehat{\Gamma}^{J_2} \partial \psi]_{-} }{\langle s-r\rangle^{\frac 12 +\delta}} \bigg \| + C_1^6 \epsilon^{\frac 52} \langle s \rangle^{-\frac 32 + 2\delta}. \label{eq:mediuesti2}
   	     \end{align}
   	     Finally, we insert the estimates \eqref{eq:mediuesti1} and \eqref{eq:mediuesti2} into \eqref{eq:derivaener}, and apply H\"older inequality, Proposition \ref{prop:L2-2d} to  obtain 
   	     \begin{align*}
   	     	E(\widehat{\Gamma}^J \partial \psi, t ) \lesssim \epsilon^2 + C_1^4 \epsilon^{\frac 52} + C_1^4 \epsilon^{\frac 94} + C_1^6 \epsilon^{\frac 52} \lesssim  \epsilon^2 +  C_1^6 \epsilon^{\frac 94}.
   	     \end{align*}
	This immediately yields \eqref{eq:2d-rough2}. The proof is completed.
	\end{proof}
	
	\begin{proof}[Proof of Proposition \ref{prop:improved-2d} and global existence] 
	Relying on  Proposition \ref{prop:improved-2d-rough}, we deduce the estimates in Proposition \ref{prop:improved-2d} by letting $C_1$ sufficiently large and $\epsilon$ very small such that $ C_1^2 \epsilon^{1/8} \ll {1/4}$.
	
	Since the energy is continuous in time, the estimates in Proposition \ref{prop:improved-2d} imply that there exists $\delta_1>0$ such that for all $t\in [t_0, T^*+\delta_1]$, it holds
	\begin{align*}
E(\widehat{\Gamma}^I \psi, t )^{1/2} &\leq C_1,
\qquad |I|\leq N,
\\
 E(\widehat{\Gamma}^J \partial \psi, t )^{1/2} &\leq C_1 \epsilon,
\qquad |J|\leq N-1.
	\end{align*}
	We know this contradicts to the definition of $T^*$ in \eqref{eq:T-2d}. Consequently $T^*=+\infty$, and thus the solution to the Cauchy problem \eqref{eq:Soler}-\eqref{eq:ID} exists globally as long as $\epsilon$ very small. 
	The pointwise estiamtes in \eqref{eq:thm-2ddecay} follows from \eqref{eq:point-2d-1}.
	\end{proof}

	\subsection{Scattering in 2D}
	To obtain scattering result for 2D cubic Dirac equation, we impose more restrictions on the nonlinear term so as to obtain more decay rate. For this purpose, we shall require 
	\begin{align*}
		H=\gamma^0,  \qquad  F=I.
	\end{align*}
	Note that  global existence of solution has been established in the above section for such a choice of $H$ and $F$. In the remaining part, we further obtain that the global solution $\psi$ scatters linearly in $H^N(\R^2)$, and finish the proof of Theorem \ref{thm:existence2d}. 
	The argument resembles that in the proof of 3D Dirac equation  with quadratic nonlinearity.
	
	\begin{proposition}
	The solution $\psi$ in Theorem \ref{thm:existence2d} scatters linearly in 2D. More precisely, there exist $\psi^+ \in H^N$ and constants $\widehat{C}>0, 0<\delta \ll 1/4$ such that 
\begin{align*}
\|\psi(t) - S(t-t_0) \psi^+\|_{H^N}
\leq
\widehat{C} \langle t\rangle^{-1/2+\delta}.
\end{align*}
	\end{proposition}
\begin{proof}
By Lemma \ref{lem:scatter}, It suffices to bound
$$
\int_{t_0}^{+\infty} \| (\psi^* \gamma^0 \psi)\psi \|_{H^N} \, d\tau.
$$
We note by Lemma \ref{lem:decompose} that
\begin{align*}
&\| (\psi^* \gamma^0 \psi)\psi \|_{H^N}
\\
\lesssim
&\sum_{|I_1|+|I_2|+|I_3|\leq N} \Big\| \big|[\nabla^{I_1} \psi]_-\big|\,    \big|\nabla^{I_2} \psi\big| \,   \big|\nabla^{I_3} \psi\big| \Big\| 
\\
\lesssim
&\sum_{\substack{|I_1|\leq N, \\|I_2|+|I_3|\leq N-3}} \bigg\| {[\nabla^{I_1} \psi]_- \over \langle\tau-r\rangle^{1/2+\delta}} \bigg\|\,    \big\|\langle\tau-r\rangle^{1/2+\delta} \nabla^{I_2} \psi\big\|_{L^\infty} \,   \big\|\nabla^{I_3} \psi \big\|_{L^\infty} 
\\
+
&\sum_{\substack{|I_1|+|I_2|\leq N-3,\\|I_3|\leq N}} \big\|[\nabla^{I_1} \psi]_-\big\|_{L^\infty}\,    \big\|\nabla^{I_2} \psi\big\|_{L^\infty} \,   \big\|\nabla^{I_3} \psi \big\| .
\end{align*}
We employ the estimates in Propositions \ref{prop:L2-2d}, \ref{prop:pointwise-2d}, and \eqref{eq:psi--point} to get
\begin{align*}
&\| (\psi^* \gamma^0 \psi)\psi \|_{H^N}
\lesssim
C_1^2 \sum_{\substack{|I_1|\leq N}} \langle \tau\rangle^{-1+\delta} \bigg\| {[\nabla^{I_1} \psi]_- \over \langle\tau-r\rangle^{1/2+\delta}} \bigg\| 
+
C_1^3 \langle \tau\rangle^{-3/2+\delta}.
\end{align*}
Thus we derive that
\begin{align*}
&\int_{t}^{+\infty} \| (\psi^* \gamma^0 \psi)\psi \|_{H^N} \, d\tau
\\
\lesssim
&C_1^3 \langle t\rangle^{-1/2+\delta} 
+ C_1^2 \Big(\int_{t}^{\infty} \langle \tau\rangle^{-2+2\delta} \, d\tau\Big)^{1/2} \sum_{\substack{|I_1|\leq N}} \Big( \int_{t}^\infty\bigg\| {\big|[\nabla^{I_1} \psi]_-\big| \over \langle\tau-r\rangle^{1/2+\delta}} \bigg\|^2 \, d\tau    \Big)^{1/2}
\\
\lesssim
&C_1^3 \langle t\rangle^{-1/2+\delta} .
\end{align*}

The proof is done.
\end{proof}

\section{Brief proof of Theorem \ref{thm:neat}}	\label{sec:neat}
	
	In this section, we verify Theorem \ref{thm:neat}. We first illustrate a proposition on the smallness of the initial data.
	
	\begin{proposition}\label{prop:ID-small}
	If the bounds in \eqref{eq:thm-ID-23} hold, then we have
	\begin{align}
	\| \langle x\rangle^{|J|} \nabla \nabla^J \psi_0\| 
	\lesssim
	\epsilon^\theta,
	\qquad
	|J|\leq N-3,
	\end{align}
	in which $\theta>0$ is a small number depending on $N$.
	\end{proposition}
	
	\begin{proof}
	The proofs for the 3D case and the 2D case are similar, but the 2D case is more subtle. So we only present the proof for the 2D case.
	
	First, by interpolation inequality we have
	\begin{align*}
		\|\nabla^J \psi_0\| \lesssim \|\psi_0\|^{1-\frac JN} \|\nabla^N \psi_0\|^{\frac JN}, \ \ 0\leq J \leq N,
	\end{align*}
	and 
	\begin{align} \label{eq:pnt-01}
	\|\nabla^I \psi_0\|_{L^{\infty}} \lesssim \|\nabla^I \psi_0\|^{\frac 12}  \|\nabla^2 \nabla^I \psi_0\|^{\frac 12}
	\lesssim \epsilon^{\frac{1+|I|}{N}}, \qquad |I|\leq N-2.
	\end{align}	
	On the other hand, we apply Klainerman-Sobolev inequality in Proposition \ref{prop:K-S} on \eqref{eq:thm-ID-23} to get
	\begin{align}\label{eq:pnt-02}
	|\nabla^I \psi_0| \lesssim \langle r\rangle^{-(|I|+1)}, \qquad |I|\leq N-2.
	\end{align}
	
	Interpolating \eqref{eq:pnt-01} and \eqref{eq:pnt-02}, we get
	\begin{align*}
		|\nabla^I \psi_0| \lesssim \epsilon^\theta \langle r\rangle^{-(|I|+1/2)}, \qquad |I|\leq N-2,
	\end{align*}
	with $\theta > 0$ a small number depending on $N$.
	Consequently, we obtain
	\begin{align*}
	\| \langle x\rangle^{|J|} \nabla \nabla^J \psi_0 \| \lesssim \epsilon^\theta, \qquad |J|\leq N-3.
	\end{align*}
	
	The proof is complete.
	\end{proof}

By the smallness of the initial data shown in Proposition \ref{prop:ID-small}, we see 	\eqref{eq:thm-ID-23} actually implies
\begin{align}\label{eq:small-01}
\sum_{0\leq |I| \leq N} \| \langle x \rangle^{|I|} \nabla^{I} \psi_0 \| < 20,
\\
\sum_{0\leq |J| \leq N-3} \| \langle x \rangle^{|J|} \nabla \nabla^J \psi_0 \| < \epsilon.
\end{align}
	This is not exactly the same as \eqref{eq:thm-3Ddata} or \eqref{eq:thm-2Ddata} due to the different regularity range, but the arguments conducted to show Theorems \ref{thm:existence}-\ref{thm:existence2d} also apply here with small modifications on the regularity range. Thus Theorem \ref{thm:neat} is verified.


	\section*{Data Availability}
The data are available upon request.

\end{document}